\documentclass[a4paper,12pt]{amsart}
\usepackage{amssymb,amsmath,amsfonts,amsthm}
\usepackage[hmargin=1in, vmargin=1.2in]{geometry}

\newcommand{\ds}{\displaystyle}

\newcommand{\NN}{\mathbb N}

\newcommand{\CC}{\mathbb C}
\newcommand{\RR}{\mathbb R}
\newcommand{\ZZ}{\mathbb Z}
\newcommand{\EE}{\mathcal E}
\newcommand{\DD}{\mathcal D}
\newcommand{\SSS}{\mathcal S}

\newcommand{\ssum}{\mbox{$\sum_j\,$}}

\newtheorem{theorem}{Theorem}[section]
\newtheorem{proposition}[theorem]{Proposition}
\newtheorem{lemma}[theorem]{Lemma}
\newtheorem{corollary}[theorem]{Corollary}

\theoremstyle{remark}
\newtheorem{remark}[theorem]{Remark}

\theoremstyle{definition}
\newtheorem{definition}[theorem]{Definition}
\newtheorem{example}[theorem]{Example}
\allowdisplaybreaks
\numberwithin{equation}{section}

\newcommand{\beq}{\begin{eqnarray}}
\newcommand{\eeq}{\end{eqnarray}}

\newcommand{\beqs}{\begin{eqnarray*}}
\newcommand{\eeqs}{\end{eqnarray*}}

\newcommand{\Op}{\mathrm{Op}}

\author[S. Pilipovi\' c]{Stevan Pilipovi\' c}
\address{Department of Mathematics and Informatics,
University of Novi Sad, Trg Dositeja Obradovi\'{c}a 4, 21000 Novi Sad, Serbia}
\email{stevan.pilipovic@dmi.uns.ac.rs}

\author[B. Prangoski]{Bojan Prangoski}
\address{Department of Mathematics, Faculty of Mechanical
Engineering-Skopje, Karposh 2 b.b., 1000 Skopje, Macedonia}
\email{bprangoski@yahoo.com}

\author[J. Vindas]{Jasson Vindas}
\thanks{The work of J. Vindas was supported by Ghent University, through the BOF-grant 01N01014.}
\address{Department of Mathematics, Ghent University, Krijgslaan 281,
9000 Ghent, Belgium}
\email{jasson.vindas@UGent.be}

\title[Spectral asymptotics for infinite order $\Psi$DOs]{Spectral asymptotics for infinite order pseudo-differential operators}



\begin{document}
\begin{abstract} We study spectral properties of a class of global infinite order pseudo-differential operators and obtain the asymptotic behaviour of the spectral counting functions of such operators. Unlike their finite order counterparts, their spectral asymptotics are not of power-log-type but of log-type. The ultradistributional setting of such operators of infinite order makes the theory more complex so that the standard finite order global Weyl calculus cannot be used in this context.
\end{abstract}
\maketitle

\section{Introduction}

In this article we study the spectral properties of global infinite order pseudo-diffe\-rential operators. Our operator classes are intrinsically related to the ultradistributional framework so that the bounds on the derivatives of the symbols are controlled by Gevrey type weight sequences.
Our aim is to establish Weyl asymptotic formulae for a large class of (hypoelliptic) $\Psi$DOs of infinite order.
 It is worth mentioning that the Weyl asymptotics for the operators that we investigate here are not of power-log-type as in the finite order (distributional) setting, but of log-type, which in turn yields that the eigenvalues of infinite order  $\Psi$DOs, with appropriate assumptions, are ``very sparse''. As a by-product of our analysis, we also obtain Weyl asymptotic formulae for a class of  finite order Shubin $\Psi$DOs with some conditions on the symbols that are not the ones usually discussed in the literature.

The spaces of symbols and corresponding pseudo-differential operators involved in this work were introduced by Prangoski  (see \cite{BojanP} for the symbolic calculus) and then extensively studied in several articles by himself and his coauthors; we refer to works of Cappiello \cite{c1,c2} for similar symbol classes related to SG-hyperbolic problems of finite order. The definition of these symbols classes is linked to two Gevrey type weight sequences $A_p$ and $M_p$, $p\in\NN$. The first one controls the smoothness, while the second one controls the growth at infinity of the symbols. These symbol classes are denoted by $\Gamma^{(M_p),\infty}_{A_p,\rho}$ and $\Gamma^{\{M_p\},\infty}_{A_p,\rho}$. The first one gives rise to operators acting continuously on Gelfand-Shilov spaces of Beurling type (i.e. of $(M_p)$-class) and the second  one on Gelfand-Shilov spaces of Roumieu type (of $\{M_p\}$-class); we will employ $\Gamma^{*,\infty}_{A_p,\rho}$ as a common notation for both cases. Since the symbols are allowed to grow sub-exponentially, i.e. ultrapolynomially, the corresponding $\Psi$DOs are of infinite order and they go beyond the classical Weyl-H\"ormander calculus.

The article is organised as follows. Section \ref{section preli}  gives some basic background material about the Gelfand-Shilov type spaces $\SSS^*(\RR^d)$ and $\SSS'^*(\RR^d)$. We collect and explain in Section \ref{sec_3}  some useful properties of the symbol classes $\Gamma^{*,\infty}_{A_p,\rho}$  and the corresponding global pseudo-differential operators. Further results related to the symbolic calculus that will be employed in the article are stated in  the Appendix (Section \ref{app}).

Section \ref{subrealisations}  is devoted to establishing  the semi-boundedness of the Weyl quantisation $a^{w}$ of a positive hypoelliptic infinite order symbol $a$. This will be achieved with the aid of results on Anti-Wick quantisation from \cite{PP}. This result is interesting by itself because hypoellipticity in this setting allows the symbols to approach $0$ sub-exponentially and thus generalises the familiar result for finite order operators. As a consequence, for hypoelliptic real-valued $a$ such that $|a(w)|\rightarrow \infty$ as $|w|\rightarrow\infty$, one obtains that the closure $\overline{A}$ of the unbounded operator $A$ on $L^2(\RR^d)$ generated by $a^w$ is self-adjoint and has a spectrum given by a sequence of eigenvalues
$\lambda_n,n\in\NN,$ tending to $\infty$ or $-\infty$, with eigenfunctions belonging to  $\SSS^*(\RR^d)$ and forming an orthonormal basis for $L^2(\RR^d)$.

We state in Section \ref{Section Weyl formulae, part I} our main results concerning Weyl asymptotic formulae and we postpone their proofs to Section \ref{proofWeylasymp}, after developing the necessary machinery. We assume there that the symbol $a$ satisfies elliptic type bounds with respect to a rather general comparison function $f$ that is positive, increasing, and has suitable growth order. Theorem \ref{Weylth1} gives the asymptotic behaviour of the spectral counting function $N(\lambda)$ for infinite order symbols, which corresponds to $f$ being of actual ultrapolynomial growth (and thus $f$ increases faster than any power function at $\infty$). Even more, our method yields new interesting results for  Shubin type $\Psi$DOs of finite order. Theorem \ref{Weylth2} deals with the case of finite order Shubin type hypoelliptic symbols that satisfy elliptic bounds but with certain growth conditions on $f$ that appear to be different from the ones treated in the literature (cf. \cite{NR,Shubin}). Theorem \ref{Weylth3} provides an $O$-bound for $N(\lambda)$ by requiring only knowledge on a lower bound for the symbol. We present there also some illustrative examples.

The heat kernel analysis needed for the proofs of the Weyl asymptotic formulae for the class of operators under consideration is given in Section \ref{Section heat parametrix}. We consider a real-valued hypoelliptic symbol $a$ in $\Gamma^{*,\infty}_{A_p,\rho}$ such that $a(w)/\ln|w|\rightarrow +\infty$ as $|w|\rightarrow\infty$. The main goal  is the analysis of the semigroup $T(t)f=\sum_{j=0}^{\infty} e^{-t\lambda_j}(f,\varphi_j)\varphi_j$, $f\in L^2(\RR^d)$, $t\geq0$, with infinitesimal generator $-\overline{A}$ (the closure of $-a^w$ in $L^2(\RR^d)$) where $\lambda_j$ and $\varphi_j$ are the eigenvalues and eigenfunctions of $\overline{A}$. The crucial result to be shown here is that $T(t)$, $t\geq 0$,  form a smooth family of operators continuously acting on $\SSS^*(\RR^d)$. The proofs of these facts are rather lengthy and we devote a whole subsection to them. It is important to stress that the classical approach does not work here (cf. Remark \ref{forthefiniteord}); one of the main reasons is the lack of Shubin-Sobolev spaces that fill in the gaps between the Gelfand-Shilov type spaces $\SSS^*(\RR^d)$ and $L^2(\RR^d)$), so we had to develop new techniques to overcome the problems. Once we have these properties of the semigroup $T(t)$, $t\geq 0$, we prove that it is equal to the heat parametrix of $a^w$ as constructed in \cite{PP1} modulo a smooth family of ultra-smoothing operators and use this to obtain the asymptotic formula
\beqs
\sum_{j=0}^{\infty} e^{-t\lambda_j}=\frac{1}{(2\pi)^d}\int_{\RR^{2d}}e^{-ta(x,\xi)}dxd\xi+ O\left(\int_{\RR^{2d}}\frac{e^{-\frac{t}{4}a(x,\xi)}}{\langle (x,\xi)\rangle^{2\rho}}dxd\xi\right),\quad t\rightarrow 0^+.
\eeqs
This key asymptotic formula is the starting point for the proofs of our main theorems from Section \ref{Section Weyl formulae, part I} concerning Weyl asymptotic formulae; such proofs are the content of Section \ref{proofWeylasymp}. The passage from asymptotics of the heat semigroup to Weyl formulae is accomplished using ideas from the theory of regular variation \cite{BGT,Korevaarbook} and Tauberian tools.

\section{Preliminaries}\label{section preli}

For $x\in \RR^d$ and
$\alpha\in\NN^d$, we will use the notation $\langle x\rangle
=(1+|x|^2)^{1/2}, D^{\alpha}= D_1^{\alpha_1}\ldots
D_d^{\alpha_d}$, where $D_j^
{\alpha_j}={i^{-\alpha_j}}\partial^{\alpha_j}/{\partial x_j}^{\alpha_j}$.
Following Komatsu \cite{Komatsu1}, we work with some of the standard conditions $(M.1)$, $(M.2)$, $(M.3)$, $(M.3)'$ and $(M.4)$ on  sequences of positive numbers $M_{p}$, $p\in\NN$, for which we always assume $M_0=1$. We only recall $(M.4):$
\\
\indent $(M.4)$ $M_p^2/p!^2\leq (M_{p-1}/(p-1)!)\cdot (M_{p+1}/(p+1)!)$, $p\in\ZZ_+$.\\
Note that the Gevrey sequence $M_p=p!^{s}$, $s>1$, satisfies all the conditions listed above. Given two weight sequences $M_p$ and $\tilde{M}_p$, the notation $M_p\subset \tilde{M}_p$ (resp. $M_p\prec\tilde{M}_p$) means that there are $C,L>0$ (resp. for every $L>0$ there is $C>0$) such that $M_p\leq CL^p\tilde{M}_p$, $\forall p\in \NN$. For a multi-index $\alpha\in\NN^d$, $M_{\alpha}$ stands for $M_{|\alpha|}$, $|\alpha|=\alpha_1+...+\alpha_d$. As usual (\cite[Section 3]{Komatsu1}), we set $m_p=M_p/M_{p-1}$, $p\in\ZZ_+$, and if $M_p$ satisfies $(M.1)$ and $M_p/C^p\rightarrow \infty$, for any $C>0$ (which obviously holds when $M_p$ satisfies $(M.3)'$), its associated function is defined by $M(\rho)=\sup  _{p\in\NN}\ln_+ \rho^{p}/M_{p}$, $\rho > 0$. It is a non-negative, continuous, monotonically increasing function, vanishes for sufficiently small $\rho>0$, and increases more rapidly than $\ln \rho^n$ as $\rho\rightarrow\infty$, for any $n\in\NN$. When $M_p=p!^{s}$, with $s>0$,  we have $M(\rho)\asymp \rho^{1/s}$.

For a regular compact set $K\subseteq \RR^d$ (i.e. $K=\overline{\mathrm{int}\, K}$) and $h>0$, $\EE^{M_p,h}(K)$ is the
Banach space (abbreviated as $(B)$-space) of all
$\varphi\in C^{\infty}(\mathrm{int}\,K)$ whose derivatives extend to continuous functions on $K$ and satisfy
$\sup_{\alpha\in\NN^d}\sup_{x\in
K}|D^{\alpha}\varphi(x)|/(h^{\alpha}M_{\alpha})<\infty$ and
$\DD^{M_p,h}_K$ denotes its subspace of all smooth functions
supported by $K$. For $U\subseteq \RR^d$, we define as locally convex spaces (abbreviated as l.c.s.)
$\EE^{(M_p)}(U),$
 $\EE^{\{M_p\}}(U),$
 $\DD^{(M_p)}(U),$
 $\DD^{\{M_p\}}(U)$ and their strong duals, the corresponding spaces of ultradistributions  of Beurling and Roumieu type, cf. \cite{Komatsu1,Komatsu2,Komatsu3}.\\
\indent We denote by $\mathfrak{R}$ the set of all positive sequences which monotonically increase to infinity. There is a natural order on $\mathfrak{R}$ defined by $(r_p)\leq (k_p)$ if $r_p\leq k_p$, $\forall p\in\ZZ_+$, and with it $(\mathfrak{R},\leq)$ becomes a directed set. For $(r_p)\in\mathfrak{R}$, consider the sequence $N_0=1$, $N_p=M_p\prod_{j=1}^{p}r_j$, $p\in\ZZ_+$. It is easy to check that this sequence satisfies $(M.1)$ and $(M.3)'$ when $M_p$ does so and its associated function will be denoted by $N_{r_p}(\rho)$, i.e. $N_{r_{p}}(\rho)=\sup_{p\in\NN} \ln_+ \rho^{p}/(M_p\prod_{j=1}^{p}r_j)$, $\rho > 0$. Note that for $(r_{p})\in\mathfrak{R}$ and $k > 0 $ there is $\rho _{0} > 0$ such that $N_{r_{p}} (\rho ) \leq M(k \rho )$, for $\rho > \rho _{0}$.\\
\indent A measurable function $f$ on $\RR^d$ is said to have
ultrapolynomial growth of class $(M_p)$ (resp. of class $\{M_p\}$)
if $\|e^{-M(h|\cdot|)}f\|_{L^{\infty}(\RR^d)}<\infty$ for some
$h>0$ (resp. for every $h>0$). We have the following equivalent description of continuous functions of ultrapolynomial growth of class $\{M_p\}$.

\begin{lemma}$($\cite[Lemma 2.1]{PP1}$)$
\label{lemulgr117}
Let $B\subseteq C(\RR^d)$. The following conditions are
equivalent: $(i)$ For every $h>0$ there exists $C>0$ such that $|f(x)|\leq Ce^{M(h|x|)}$, for all $x\in\RR^d$, $f\in B$;
$\;(ii)$ There exist $(r_p)\in\mathfrak{R}$ and $C>0$ such that $|f(x)|\leq Ce^{N_{r_p}(|x|)}$, for all $x\in\RR^d$, $f\in B$.
\end{lemma}

\indent If $M_p$ satisfies $(M.1)$ and $(M.3)'$, for $m>0$ we
denote by $\SSS^{M_p,m}_{\infty}(\RR^d)$ the $(B)$-space of all
$\varphi\in C^{\infty}(\RR^d)$ for which the norm
 $\sup_{\alpha\in
\NN^d}m^{|\alpha|}\|e^{M(m|\cdot|)}D^{\alpha}\varphi\|_{L^{\infty}(\RR^d)}/M_{\alpha}$
is finite. The spaces of sub-exponentially decreasing
ultradifferentiable function of Beurling and Roumieu type are
defined as
 $$
\SSS^{(M_{p})}(\RR^d)=\lim_{\substack{\longleftarrow\\
m\rightarrow\infty}}\SSS^{M_{p},m}_{\infty}\left(\RR^d\right)\quad
\mbox{and}\quad
\SSS^{\{M_{p}\}}(\RR^d)=\lim_{\substack{\longrightarrow\\
m\rightarrow 0}}\SSS^{M_{p},m}_{\infty}\left(\RR^d\right),
$$
respectively. Their strong duals $\SSS'^{(M_{p})}(\RR^d)$ and $\SSS'^{\{M_{p}\}}(\RR^d)$ are the spaces of tempered ultradistributions of Beurling and Roumieu type, respectively. When $M_p=p!^{s}$, $s>1$, $\SSS^{\{M_p\}}(\RR^d)$ is just the Gelfand-Shilov space $\SSS^{s}_{s}(\RR^d)$ \cite{NR}. If $M_p$ satisfies $(M.2)$, the ultradifferential operators of class $*$ act continuously on $\SSS^*(\RR^d)$ and $\SSS'^*(\RR^d)$ (for the definition of ultradifferential operators see \cite{Komatsu1}). These spaces are nuclear and the Fourier transform is a topological isomorphism on them. We refer to \cite{PilipovicK,PilipovicU} for the topological properties of $\SSS^*(\RR^d)$ and $\SSS'^*(\RR^d)$. Here we recall that, when $M_p$ satisfies $(M.2)$, the space $\SSS^{\{M_p\}}(\RR^d)$ is topologically isomorphic to $\ds\lim_{\substack{\longleftarrow\\ (r_p)\in\mathfrak{R}}}\SSS^{M_p, (r_p)}_{\infty}(\RR^d)$, where the projective limit is taken with respect to the natural order on $\mathfrak{R}$ defined above and $\SSS^{M_p, (r_p)}_{\infty}(\RR^d)$ is the $(B)$-space of all $\varphi\in C^{\infty}(\RR^d)$ for which the norm $\sup_{\alpha\in \NN^d}\|e^{N_{r_p}(|\cdot|)}D^{\alpha}\varphi\|_{L^{\infty}(\RR^d)}/ (M_{\alpha}\prod_{j=1}^{|\alpha|}r_j)$ is finite.\\
\indent Next, let $E$ and $F$ be  l.c.s.;  $\mathcal{L}(E,F)$ stands
for the space of continuous linear mappings from $E$ to $F$; when
$E=F$, we write $\mathcal{L}(E)$. We employ the notation $\mathcal{L}_b(E,F)$ for the
space $\mathcal{L}(E,F)$ equipped with the topology of bounded
convergence and, similarly, $\mathcal{L}_p(E,F)$ and $\mathcal{L}_{\sigma}(E,F)$ stand for
$\mathcal{L}(E,F)$ equipped with the topologies of precompact and simple convergence, respectively. Furthermore, $E\hookrightarrow F$ means that $E$ is continuously and densely included in $F$. For $(a,b)\subseteq\RR$ and $0\leq k\leq
\infty$,  $ C^{k}((a,b);E)$ stands for the vector space of $k$ times continuously differentiable $E$-valued functions on $(a,b)$, while $ C^{k}([a,b);E)$ for the space of those on $[a,b)$, where the
derivatives at $a$ are to be understood as right derivatives; we
use analogous notations when considering functions over $(a,b]$ or
$[a,b]$.

\section{$\Psi$DOs of infinite order of Shubin type on $\SSS^*(\RR^d)$ and $\SSS'^*(\RR^d)$}\label{sec_3}
We discuss in this section properties of the classes of infinite order $\Psi$DOs that we shall consider in the article; see also the Appendix for other important facts about their  symbolic calculus. We refer to \cite{BojanP,CPP} and \cite[Sections 3 and 4]{PP1} for complete accounts.

\subsection{Symbol classes and symbolic calculus}\label{sub calculus}
Let $A_p$ and $M_p$ be two weight sequences of positive numbers such that $A_0=A_1=M_0=M_1=1$. We assume that $M_p$ satisfies $(M.1)$, $(M.2)$ and $(M.3)$, and that $A_p$ satisfies $(M.1)$, $(M.2)$, $(M.3)'$ and $(M.4)$. Of course, we may assume that the constants $c_0$ and $H$ appearing in $(M.2)$ are the same for both sequences $M_p$ and $A_p$. We assume that $A_p\subset M_p$. Let $\rho_0=\inf\{\rho\in\RR_+|\,A_p\subset M_p^{\rho}\}$; clearly $0<\rho_0\leq 1$. Throughout the rest of the article, $\rho$ is a fixed number satisfying $\rho_0\leq \rho\leq1$, if the infimum is reached, or, otherwise $\rho_0< \rho\leq1$. Clearly, we may also assume that $A_p\leq c_0 L^pM_p^{\rho}$, where $c_0\geq 1$ is the constant from $(M.2)$.\\
\indent For $h,m>0$, define (following \cite{BojanP})
$\Gamma_{A_p,\rho}^{M_p,\infty}(\RR^{2d};h,m)$ to be the
$(B)$-space of all $a\in  C^{\infty}(\RR^{2d})$ for which
the norm
\beqs
\sup_{\alpha,\beta\in\NN^d}\sup_{(x,\xi)\in\RR^{2d}}\frac{\left|D^{\alpha}_{\xi}D^{\beta}_x
a(x,\xi)\right| \langle
(x,\xi)\rangle^{\rho|\alpha|+\rho|\beta|}e^{-M(m|\xi|)}e^{-M(m|x|)}}
{h^{|\alpha|+|\beta|}A_{\alpha}A_{\beta}}.
\eeqs
is finite. As l.c.s., we define
$$
\Gamma_{A_p,\rho}^{(M_p),\infty}(\RR^{2d};m)=
\lim_{\substack{\longleftarrow\\h\rightarrow 0}}
\Gamma_{A_p,\rho}^{M_p,\infty}(\RR^{2d};h,m);
\quad
\Gamma_{A_p,\rho}^{(M_p),\infty}(\RR^{2d})
\lim_{\substack{\longrightarrow\\m\rightarrow\infty}}
\Gamma_{A_p,\rho}^{(M_p),\infty}(\RR^{2d};m);
$$
$$
\Gamma_{A_p,\rho}^{\{M_p\},\infty}(\RR^{2d};h)=
\lim_{\substack{\longleftarrow\\m\rightarrow 0}}
\Gamma_{A_p,\rho}^{M_p,\infty}(\RR^{2d};h,m);
\quad
\Gamma_{A_p,\rho}^{\{M_p\},\infty}(\RR^{2d})=
\lim_{\substack{\longrightarrow\\h\rightarrow\infty}}
\Gamma_{A_p,\rho}^{\{M_p\},\infty}(\RR^{2d};h).
$$
Then,
$\Gamma_{A_p,\rho}^{(M_p),\infty}(\RR^{2d};m)$ and
$\Gamma_{A_p,\rho}^{\{M_p\},\infty}(\RR^{2d};h)$ are $(F)$-spaces.
The spaces $\Gamma_{A_p,\rho}^{*,\infty}(\RR^{2d})$ are barrelled
and bornological.\\

\indent For $\tau\in\RR$ and
$a\in\Gamma_{A_p,\rho}^{*,\infty}(\RR^{2d})$, the
$\tau$-quantisation of $a$ is the operator $\Op_{\tau}(a)$, continuous on
 $\SSS^*(\RR^d)$
given by the  iterated integral:
\beqs
\left(\Op_{\tau}(a)u\right)(x)=\frac{1}{(2\pi)^d}\int_{\RR^d}\int_{\RR^d}e^{i(x-y)\xi}a((1-\tau)x+\tau
y,\xi) u(y)dyd\xi.
\eeqs

Let $t\geq0$. We denote $Q_t=\left\{(x,\xi)\in\RR^{2d}|\,\langle
x\rangle<t, \langle \xi\rangle<t\right\}$ and
$Q_t^c=\RR^{2d}\backslash Q_t$. If $0\leq t\leq 1$, then
$Q_t=\emptyset$ and $Q_t^c=\RR^{2d}$. Let $B\geq 0$ and $h,m>0$.
Following \cite{BojanP,PP1}, denote by $FS_{A_p,\rho}^{M_p,\infty}(\RR^{2d};B,h,m)$ the vector
space of all formal series $\sum_{j=0}^{\infty}a_j(x,\xi)$ such
that $a_j\in  C^{\infty}(\mathrm{int\,}Q^c_{Bm_j})$,
$D^{\alpha}_{\xi} D^{\beta}_x a_j(x,\xi)$ can be extended to a
continuous function on $Q^c_{Bm_j}$ for all $\alpha,\beta\in\NN^d$
and
\beqs
\sup_{j\in\NN}\sup_{\alpha,\beta}\sup_{(x,\xi)\in
Q_{Bm_j}^c}\frac{\left|D^{\alpha}_{\xi}D^{\beta}_x
a_j(x,\xi)\right| \langle
(x,\xi)\rangle^{\rho|\alpha|+\rho|\beta|+2j\rho}e^{-M(m|\xi|)}e^{-M(m|x|)}}
{h^{|\alpha|+|\beta|+2j}A_{\alpha}A_{\beta}A_jA_j}<\infty.
\eeqs
In the above, we use the convention $m_0=0$ and hence,
$Q^c_{Bm_0}=\RR^{2d}$. With this norm,
$FS_{A_p,\rho}^{M_p,\infty}\left(\RR^{2d};B,h,m\right)$ becomes a
$(B)$-space. As l.c.s., we define
\beqs
FS_{A_p,\rho}^{(M_p),\infty}(\RR^{2d};B,m)&=&\lim_{\substack{\longleftarrow\\h\rightarrow
0}}
FS_{A_p,\rho}^{M_p,\infty}(\RR^{2d};B,h,m),\\
FS_{A_p,\rho}^{(M_p),\infty}(\RR^{2d};B)&=&
\lim_{\substack{\longrightarrow\\m\rightarrow\infty}}
FS_{A_p,\rho}^{(M_p),\infty}(\RR^{2d};B,m),\\
FS_{A_p,\rho}^{\{M_p\},\infty}(\RR^{2d};B,h)&=&
\lim_{\substack{\longleftarrow\\m\rightarrow 0}}
FS_{A_p,\rho}^{M_p,\infty}(\RR^{2d};B,h,m),\\
FS_{A_p,\rho}^{\{M_p\},\infty}(\RR^{2d};B)&=&\lim_{\substack{\longrightarrow\\h\rightarrow
\infty}} FS_{A_p,\rho}^{\{M_p\},\infty}(\RR^{2d};B,h).
\eeqs
Then, the spaces $FS_{A_p,\rho}^{(M_p),\infty}(\RR^{2d};B,m)$ and $FS_{A_p,\rho}^{\{M_p\},\infty}(\RR^{2d};B,h)$ are $(F)$-spaces and the space $FS_{A_p,\rho}^{*,\infty}(\RR^{2d};B)$ is barrelled and bornological. The inclusion mapping $\Gamma_{A_p,\rho}^{*,\infty}(\RR^{2d})\rightarrow$ $FS_{A_p,\rho}^{*,\infty}(\RR^{2d};B)$,
defined as $a\mapsto\sum_{j\in\NN}a_j$, where $a_0=a$ and $a_j=0$, $j\geq 1$, is continuous. We call this inclusion the canonical one. For $B_1\leq B_2$, the mapping $\sum_j p_j\mapsto \sum_j p_{j|_{Q_{B_2m_j}^c}}$, $FS_{A_p,\rho}^{*,\infty}(\RR^{2d};B_1)\rightarrow FS_{A_p,\rho}^{*,\infty}(\RR^{2d};B_2)$ is continuous. We also call this mapping canonical.\\
\indent Let $\ds FS_{A_p,\rho}^{*,\infty}(\RR^{2d})=\lim_{\substack{\longrightarrow\\ B\rightarrow\infty}}FS_{A_p,\rho}^{*,\infty}(\RR^{2d};B)$, where the inductive limit is taken in an algebraic sense and the linking mappings are the canonical ones described above. Clearly, $FS_{A_p,\rho}^{*,\infty}(\RR^{2d})$ is non-trivial.\\
\indent If $\sum_j a_j\in FS_{A_p,\rho}^{*,\infty}(\RR^{2d};B)$
and $n\in\NN$, $(\sum_j a_j)_n$ will just mean the function
$a_n\in C^{\infty}(Q_{Bm_n}^c)$, while
$(\sum_j a_j)_{<n}$ denotes the function $\sum_{j=0}^{n-1}
a_j\in C^{\infty}(Q_{Bm_{n-1}}^c)$. Furthermore,
$\mathbf{1}$ denotes the element $\sum_j a_j\in
FS_{A_p,\rho}^{*,\infty}(\RR^{2d};B)$ given by $a_0(x,\xi)=1$ and
$a_j(x,\xi)=0$, $j\in\ZZ_+$.

Recall, \cite[Definition 3]{BojanP} that two sums, $\sum_{j\in\NN}a_j,\,\sum_{j\in\NN}b_j\in
FS_{A_p,\rho}^{*,\infty}(\RR^{2d})$, are said to be equivalent, in
notation $\sum_{j\in\NN}a_j\sim\sum_{j\in\NN}b_j$, if there exist
$m>0$ and $B>0$ (resp. there exist $h>0$ and $B>0$), such that for
every $h>0$ (resp. for every $m>0$),
\beqs
\sup_{n\in\ZZ_+}\sup_{\alpha,\beta}\sup_{(x,\xi)\in
Q_{Bm_n}^c}\frac{\left|D^{\alpha}_{\xi}D^{\beta}_x
\sum_{j<n}\left(a_j(x,\xi)-b_j(x,\xi)\right)\right| \langle
(x,\xi)\rangle^{\rho|\alpha|+\rho|\beta|+2n\rho}}
{h^{|\alpha|+|\beta|+2n}A_{\alpha}A_{\beta}A_nA_n e^{M(m|\xi|)}e^{M(m|x|)}}<\infty.
\eeqs

\subsection{Subordination}
In the sequel, we will often use the notation
$w=(x,\xi)\in\RR^{2d}$.\\
\indent Let $\Lambda$ be an index set and $\{f_{\lambda}|\, \lambda\in\Lambda\}$ be a set of positive continuous functions on $\RR^{2d}$ each with ultrapolynomial growth of class $*$. We say that a set $U^{(\Lambda)}=\left\{\sum_j a^{(\lambda)}_j\big|\, \lambda\in\Lambda\right\}\subseteq FS_{A_p,\rho}^{*,\infty}(\RR^{2d};B')$ is subordinated to $\{f_{\lambda}|\, \lambda\in\Lambda\}$ in $FS_{A_p,\rho}^{*,\infty}(\RR^{2d})$, in notation $U^{(\Lambda)}\precsim \{f_{\lambda}|\, \lambda\in\Lambda\}$, if the following estimate holds: there exists $B\geq B'$ such that for every $h>0$ there
exists $C>0$ (resp. there exist $h,C>0$) such that
\beqs
\sup_{\lambda\in\Lambda}\sup_{j\in\NN}\sup_{\alpha\in\NN^{2d}}\sup_{w\in
Q_{Bm_j}^c}\frac{\left|D^{\alpha}_w
a^{(\lambda)}_j(w)\right|\langle
w\rangle^{\rho(|\alpha|+2j)}}{h^{|\alpha|+2j}A_{|\alpha|+2j}f_{\lambda}(w)}\leq
C.
\eeqs
Whenever we want to emphasise that the estimate is valid for a particular $B\geq B'$, we write $U^{(\Lambda)}\precsim \{f_{\lambda}|\, \lambda\in\Lambda\}$ in $FS_{A_p,\rho}^{*,\infty}(\RR^{2d};B)$.
When $f_{\lambda}=f$, $\forall \lambda\in\Lambda$, we abbreviate the notation and simply write $U\precsim f$, and then say that $U$ is subordinated to $f$. Clearly, for $U\subseteq FS_{A_p,\rho}^{*,\infty}(\RR^{2d};B_1)$ such that $U\precsim f$, there exists $B\geq B_1$ such that the image of $U$ under the canonical mapping $FS_{A_p,\rho}^{*,\infty}(\RR^{2d};B_1)\rightarrow FS_{A_p,\rho}^{*,\infty}(\RR^{2d};B)$ is a bounded subset of $FS_{A_p,\rho}^{(M_p),\infty}(\RR^{2d};B,m)$ for some $m>0$ (resp. a bounded subset of $FS_{A_p,\rho}^{\{M_p\},\infty}(\RR^{2d};B,h)$ for some $h>0$). For such $U$,
 we say that a bounded set $V$ in $\Gamma_{A_p,\rho}^{(M_p),\infty}(\RR^{2d};m)$ for some $m>0$ (resp. in $\Gamma_{A_p,\rho}^{\{M_p\},\infty}(\RR^{2d};h)$ for some $h>0$) is subordinated to $U$ under $f$, in notations $V\precsim_f U$, if there exists a surjective mapping $\Sigma:U\rightarrow V$ such that the following estimate holds: there exists $B\geq B_1$ such that for every $h>0$ there
exists $C>0$ (resp. there exist $h,C>0$) such that for all $\sum_j
a_j\in U$ and the corresponding $\Sigma(\sum_j a_j)=a\in V$
\beqs
\sup_{n\in\ZZ_+}\sup_{\alpha\in\NN^{2d}}\sup_{w\in
Q_{Bm_n}^c}\frac{\left|D^{\alpha}_w\left(a(w)-
\sum_{j<n}a_j(w)\right)\right|\langle
w\rangle^{\rho(|\alpha|+2n)}}{h^{|\alpha|+2n}A_{|\alpha|+2n}f(w)}\leq
C.
\eeqs
Again, when we want to emphasise the particular $B$ for which this holds, we write $V\precsim_f U$ in $FS_{A_p,\rho}^{*,\infty}(\RR^{2d};B)$.
If $V\precsim_f U$ and if
we denote by $\tilde{V}$ the image of $V$ under the canonical
inclusion $\Gamma_{A_p,\rho}^{*,\infty}(\RR^{2d})\rightarrow
FS_{A_p,\rho}^{*,\infty}(\RR^{2d};0)$, $a\mapsto
a+\sum_{j\in\ZZ_+}0$, then by specialising the above estimate for
$n=1$ together with the boundedness of $V$ in $\Gamma_{A_p,\rho}^{(M_p),\infty}(\RR^{2d};m)$ for some $m>0$
(resp. in $\Gamma_{A_p,\rho}^{\{M_p\},\infty}(\RR^{2d};h)$ for
some $h>0$) and the continuity and positivity of $f$, we derive that
$\tilde{V}\precsim f$ in $FS_{A_p,\rho}^{*,\infty}(\RR^{2d};0)$.
In such a case, we slightly abuse notation and write $V\precsim f$.
This estimate also implies $\Sigma(\sum_j a_j)\sim\sum_j a_j$. To
see that given such an $U\subseteq
FS_{A_p,\rho}^{*,\infty}(\RR^{2d};B)$ there always exists
$V\precsim_f U$, we can proceed as follows. Let
$\psi\in\DD^{(A_p)}(\RR^{d})$ in the $(M_p)$ case and
$\psi\in\DD^{\{A_p\}}(\RR^{d})$ in the $\{M_p\}$ case
respectively, such that $0\leq \psi\leq 1$, $\psi(\xi)=1$ when
$\langle\xi\rangle\leq 2$ and $\psi(\xi)=0$ when
$\langle\xi\rangle\geq 3$. Set $\chi(x,\xi)=\psi(x)\psi(\xi)$,
$\chi_{n,R}(w)=\chi(w/(Rm_n))$ for $n\in\ZZ_+$ and $R>0$ and put
$\chi_{0,R}(w)=0$. Given $U\subseteq
FS_{A_p,\rho}^{*,\infty}(\RR^{2d};B)$ as above, for $\sum_j a_j\in
U$ denote $R(\sum_j a_j)(w)= \sum_{j=0}^{\infty}
(1-\chi_{j,R}(w))a_j(w)$. If $R> B$, this is a well defined smooth function on $\RR^{2d}$, since the series is locally finite.

\begin{proposition}\label{subordinate}$($\cite[Proposition 3.3]{PP1}$)$
Let $U=\left\{\sum_j a^{(\lambda)}_j\big|\, \lambda\in\Lambda\right\}$ be a subset of $FS_{A_p,\rho}^{*,\infty}(\RR^{2d};B')$ that is subordinated to $\{f_{\lambda}|\, \lambda\in\Lambda\}$ in $FS_{A_p,\rho}^{*,\infty}(\RR^{2d})$. There exists $R_0> B'$ such that for each $R\geq R_0$, $U_R=\left\{R(\sum_j a^{(\lambda)}_j)\big|\, \lambda\in\Lambda\right\}\subseteq \Gamma_{A_p,\rho}^{*,\infty}(\RR^{2d})$ and the following estimate holds: there exists $B=B(R)\geq B'$ such that for every $h>0$
there exists $C>0$ (resp. there exist $h,C>0$) such that
\beqs
\sup_{\lambda\in\Lambda}\sup_{n\in\ZZ_+}\sup_{\alpha\in\NN^{2d}}\sup_{w\in
Q_{Bm_n}^c}\frac{\left|D^{\alpha}_w\left(R(\sum_ja^{(\lambda)}_j)(w)-
\sum_{j<n}a^{(\lambda)}_j(w)\right)\right|\langle
w\rangle^{\rho(|\alpha|+2n)}}{h^{|\alpha|+2n}A_{|\alpha|+2n}f_{\lambda}(w)}\leq
C.
\eeqs
If in addition $f_{\lambda}=f$, $\forall
\lambda\in\Lambda$, then $U_R$ is bounded in
$\Gamma_{A_p,\rho}^{(M_p),\infty}(\RR^{2d};m)$ for some $m>0$
(resp. bounded in $\Gamma_{A_p,\rho}^{\{M_p\},\infty}(\RR^{2d};h)$
for some $h>0$) and hence $U_R\precsim_f U$.
\end{proposition}

We say that this $U_{R}$ is canonically obtained from $U$ by
$\{\chi_{n,R}\}_{n\in\NN}$. Of course, here the mapping
$\Sigma: U\rightarrow U_R$ is just $\sum_j a_j\mapsto R(\sum_j
a_j)$.

\begin{proposition}\label{eqsse}$($\cite[Proposition 3.4]{PP1}$)$
Let $V$ be a bounded subset of
$\Gamma_{A_p,\rho}^{(M_p),\infty}(\RR^{2d};\tilde{m})$ for some
$\tilde{m}>0$ (resp. of
$\Gamma_{A_p,\rho}^{\{M_p\},\infty}(\RR^{2d};\tilde{h})$ for some
$\tilde{h}>0$). Assume that there exist $B,m>0$ such that for
every $h>0$ there exists $C>0$ (resp. there exist $B,h>0$ such
that for every $m>0$ there exists $C>0$) such that
\beqs
\sup_{a\in V}\sup_{n\in\ZZ_+}\sup_{\alpha\in\NN^{2d}}\sup_{w\in
Q_{Bm_n}^c}\frac{\left|D^{\alpha}_w a(w)\right|\langle
w\rangle^{\rho(|\alpha|+2n)}}{h^{|\alpha|+2n}A_{|\alpha|+2n}e^{M(m|w|)}}\leq
C.
\eeqs
Then, $\{\mathrm{Op}_{\tau}(a)|\,
a\in U\}$ is an equicontinuous subset of
$\mathcal{L}(\SSS'^*(\RR^d),\SSS^*(\RR^d))$ for each $\tau\in\RR$.
\end{proposition}

In what follows, we will frequently use the term ``$*$-regularising set''
for a subset of $\mathcal{L}(\SSS'^*(\RR^d),\SSS^*(\RR^d))$. Changing the quantisation and taking composition of $\Psi$DOs with symbols in $\Gamma_{A_p,\rho}^{*,\infty}(\RR^{2d})$ always results in $\Psi$DOs with symbols in the same class modulo $*$-regularising operators; we collect some of these facts in the Appendix and we refer to \cite{PP1,BojanP} for the complete theory.

\subsection{Weyl quantisation. The sharp product in  $FS_{A_p,\rho}^{*,\infty}(\RR^{2d};B)$}\label{sub Weyl quantasation}

We recall in this and the next subsection results from \cite{PP1}
about the Weyl quantisation of symbols; we often write $a^w$ instead of
$\Op_{1/2}(a)$.

Given $\sum_j a_j,\sum_j b_j\in
FS_{A_p,\rho}^{*,\infty}(\RR^{2d};B)$ we define their sharp
product, denoted as $\sum_j a_j \# \sum_j b_j$, via the formal series $\sum_j
c_j=\sum_j a_j \# \sum_j b_j$ where
\beqs
c_j(x,\xi)=\sum_{s+k+l=j}\sum_{|\alpha+\beta|=l}\frac{(-1)^{|\beta|}} {\alpha!\beta!2^l}\partial^{\alpha}_{\xi}D^{\beta}_x
a_s(x,\xi)\partial^{\beta}_{\xi} D^{\alpha}_x b_k(x,\xi),\,\,
(x,\xi)\in Q^c_{Bm_j}.
\eeqs
It is easy to verify that $\sum_j c_j$
is a well defined element of $FS_{A_p,\rho}^{*,\infty}(\RR^{2d};B)$.
If $a\in \Gamma_{A_p,\rho}^{*,\infty}(\RR^{2d})$, then $a\#\sum_j
b_j$ will denote the $\#$ product of the image of $a$ under the
canonical inclusion
$\Gamma_{A_p,\rho}^{*,\infty}(\RR^{2d})\rightarrow
FS_{A_p,\rho}^{*,\infty}(\RR^{2d};B)$ and $\sum_j b_j$. The same convention applies if $b\in \Gamma_{A_p,\rho}^{*,\infty}(\RR^{2d})$ or if both
$a,b\in \Gamma_{A_p,\rho}^{*,\infty}(\RR^{2d})$.

\begin{remark}\label{real-valuedsym}
If $\sum_j a_j,\sum_j b_j\in FS_{A_p,\rho}^{*,\infty}(\RR^{2d};B)$ and $\sum_j c_j=\sum_j a_j\#\sum_jb_j$, then $\sum_j\overline{c_j}=\sum_j\overline{b_j}\#\sum_j\overline{a_j}$. In particular, if $a_j$ and $b_j$ are real-valued for all $j\in\NN$ and $\sum_ja_j\#\sum_jb_j=\sum_jb_j\#\sum_ja_j$, then $c_j$ are real-valued for all $j\in\NN$.
\end{remark}

\begin{proposition}\label{hyporingg}$($\cite[Proposition 4.5]{PP1}$)$
For each $B\geq 0$, $FS_{A_p,\rho}^{*,\infty}(\RR^{2d};B)$ is a
ring with the pointwise addition and multiplication given by $\#$.
Moreover, the multiplication
$\#:FS_{A_p,\rho}^{*,\infty}(\RR^{2d};B)\times
FS_{A_p,\rho}^{*,\infty}(\RR^{2d};B)\rightarrow
FS_{A_p,\rho}^{*,\infty}(\RR^{2d};B)$ is hypocontinuous.
\end{proposition}

The multiplicative identity of $FS_{A_p,\rho}^{*,\infty}(\RR^{2d};B)$ is given by $\mathbf{1}$. The $\#$-product of symbols corresponds to the composition of their Weyl quantisation (see the Appendix).

\section{Hypoelliptic operators of infinite order}\label{subrealisations}

This section is devoted to hypoellipticity in the context of our symbol classes. Our main goal below is to establish a semi-boundedness result. In preparation, we start by discussing
$L^2$-realisations of the associated unbounded operators.

\begin{lemma}\label{operatoronl2}$($\cite[Lemma 5.3]{PP1}$)$
Let $V\subseteq \Gamma_{A_p,\rho}^{*,\infty}(\RR^{2d})$. Assume that for every $h>0$ there exists $C>0$ (resp. there exist $h,C>0$) such that
\beq\label{estforope}
\left|D^{\alpha}_w b(w)\right|\leq Ch^{|\alpha|}A_{\alpha}\langle w\rangle^{-\rho|\alpha|},\,\, w\in\RR^{2d},\,\alpha\in \NN^{2d},\, b\in V.
\eeq
Then, for each $b\in V$, $b^w$ extends to a bounded operator on $L^2(\RR^d)$ and the set $\{b^w|\, b\in V\}$ is bounded in $\mathcal{L}_b(L^2(\RR^d),L^2(\RR^d))$. If $\{b_{\lambda}\}_{\lambda\in\Lambda}\subseteq V$ is a net that converges to $b_0\in V$ in the topology of $\Gamma_{A_p,\rho}^{*,\infty}(\RR^{2d})$, then $b_{\lambda}^w\rightarrow b_0^w$ in $\mathcal{L}_p(L^2(\RR^d),L^2(\RR^d))$.
\end{lemma}

Given $a\in\Gamma^{*,\infty}_{A_p,\rho}(\RR^{2d})$, let us denote by $A$
the unbounded operator on $L^2(\RR^d)$ with domain $\SSS^*(\RR^d)$
defined as $A \varphi=a^w\varphi$, $\varphi\in\SSS^*(\RR^d)$. Considering
$a^w$ as a mapping on  $\SSS'^*(\RR^d),$ its restriction   to the subspace $\{g\in L^2(\RR^d)|\, a^wg\in
L^2(\RR^d)\}$ defines a closed extension of $A$ which is called
the maximal realisation of $A$. As standard, we denote by $\overline{A}$ the
closure of $A$, also called the minimal realisation of $A$. Notice that the formal adjoint $(a^w)^*$ is in fact the pseudo-differential operator
$\bar{a}^w$ and hence, it can be extended to a continuous operator
on $\SSS'^*(\RR^d)$. One can also consider
the adjoint $A^*$ of $A$ in $L^2(\RR^d)$. The following result
gives the precise connection between $A^*$ and $(a^w)^*$. Its proof is
completely analogous to the one in the classical case for finite order $\Psi$DOs and we omit it (see for example \cite[Proposition 4.2.1,
p. 160]{NR}).

\begin{proposition}\label{formal_adj}
Let $a\in\Gamma_{A_p,\rho}^{*,\infty}(\RR^{2d})$ with $A$ and
$A^*$ defined as above. Then $A^*$ coincides with the maximal
realisation of $(a^w)^*$, i.e. the domain of $A^*$ is
$D(A^*)=\{g\in L^2(\RR^d)|\, (a^w)^*g\in L^2(\RR^d)\}$ and
$A^*g=(a^w)^*g$, $\forall g\in D(A^*)$.
\end{proposition}

We now introduce the notion of hypoellipticity in $\Gamma^{*,\infty}_{A_p,\rho}$.
\begin{definition}(\cite[Definition 1.1]{CPP})
Let $a\in\Gamma^{*,\infty}_{A_p,\rho}(\RR^{2d})$. We say that $a$
is $\Gamma^{*,\infty}_{A_p,\rho}$-hypoelliptic (or, in short, simply hypoelliptic) if
\begin{itemize}
\item[$i$)] there exists $B>0$ such that there are $c,m>0$ (resp. for every $m>0$ there is $c>0$) such that
\beq\label{dd1}
|a(x,\xi)|\geq c e^{-M(m|x|)-M(m|\xi|)},\quad
(x,\xi)\in Q^c_B,
\eeq
\item[$ii$)] there exists $B>0$ such that for every $h>0$ there is $C>0$ (resp. there are $h,C>0$) such that
\beq\label{dd2}
\left|D^{\alpha}_{\xi}D^{\beta}_x
a(x,\xi)\right|\leq
C\frac{h^{|\alpha|+|\beta|}|a(x,\xi)|A_{\alpha}A_{\beta}}{\langle(x,\xi)\rangle^{\rho(|\alpha|+|\beta|)}},\,\,
\alpha,\beta\in\NN^d,\, (x,\xi)\in Q^c_B.
\eeq
\end{itemize}
\end{definition}

Operators with hypoelliptic symbols have parametrices and hence are globally regular; see the Appendix for the precise results.

\begin{proposition}\label{maximalreal}$($\cite[Proposition 5.4]{PP1}$)$
Let $a$ be hypoelliptic and $A$ be the corresponding unbounded
operator on $L^2(\RR^d)$ defined above. Then the minimal
realisation $\overline{A}$ coincides with the maximal realisation.
Moreover, $\overline{A}$ coincides with the restriction of $a^w$
on the domain of $\overline{A}$. If additionally $a$ is real-valued, then $\overline{A}$ is a self-adjoint operator on $L^2(\RR^d)$.
\end{proposition}

\subsection{Semi-boundedness and the spectrum of operators with positive hypoelliptic Weyl symbols}\label{sub semi-boundedness}

Before we can say anything meaningful about the spectrum of operators with hypoelliptic positive Weyl symbols, we need to prove that such operators are always semi-bounded. This is a well know fact for finite order symbols. We prove here that it remains true even in the infinite order case. In order to appreciate more this result, the reader should keep in mind the operators can be of truly infinite order, i.e. the symbols are allowed to have ultrapolynomial growth; such operators then go beyond the classical Weyl-H\"ormander calculus.

\begin{proposition}\label{semi-bound}
Let $b\in\Gamma_{A_p,\rho}^{*,\infty}(\RR^{2d})$ be positive
hypoelliptic symbol. Then, there exists $C>0$ such that
$(b^w\varphi,\varphi)\geq -C\|\varphi\|^2_{L^2(\RR^d)}$, $\forall
\varphi\in\SSS^*(\RR^d)$.
\end{proposition}

\begin{proof} The proof heavily relies on the connection between the Weyl and the anti-Wick quantisation of symbols from $\Gamma_{A_p,\rho}^{*,\infty}(\RR^{2d})$ (see \cite{PP}). For $a\in \Gamma_{A_p,\rho}^{*,\infty}(\RR^{2d})$, we denote by $A_a$ its anti-Wick quantisation. By \cite[Theorem 3.2]{PP}, there exists
$a\in\Gamma_{A_p,\rho}^{*,\infty}(\RR^{2d})$ and a
$*$-regularising operator $T$ such that $b^w=A_a+T$. By a careful
inspection of  the proof of the quoted result, one can find the explicit
construction of $a$; it is given as follows. Start with $p'_{k,j}\in C^{\infty}(\RR^{2d})$, $k,j\in\NN$, defined
by $p'_{0,0}=b$, $p'_{k,0}=0$ for all $k\in\ZZ_+$, $p'_{k,j}=0$ for
all $0\leq k<j$, and
\beqs
p'_{k,j}(x,\xi)&=&\sum_{\substack{l_1+\ldots+l_j=k\\ l_1\geq
1,\ldots, l_j\geq
1}}\sum_{|\alpha^{(1)}+\beta^{(1)}|=2l_1,\ldots,|\alpha^{(j)}+\beta^{(j)}|=2l_j}
\frac{c_{\alpha^{(1)},\beta^{(1)}}\cdot\ldots\cdot
c_{\alpha^{(j)},\beta^{(j)}}}{\alpha^{(1)}!\beta^{(1)}!\cdot\ldots\cdot
\alpha^{(j)}!\beta^{(j)}!}\\
&{}&\cdot\partial^{\alpha^{(1)}+\ldots+ \alpha^{(j)}}_{\xi}\partial^{\beta^{(1)}+\ldots+\beta^{(j)}}_xb(x,\xi),
\eeqs
for all $x,\xi\in\RR^d$, $k\geq j$, where
$c_{\alpha,\beta}=\pi^{-d}\int_{\RR^{2d}}\eta^{\alpha}y^{\beta}e^{-|y|^2-|\eta|^2}dyd\eta$,
$\alpha,\beta\in\NN^d$. Since $b$ is positive and hypoelliptic, the estimate (\ref{dd2}) holds on the whole $\RR^{2d}$ for $b$. Repeating the proof of \cite[Theorem
3.2]{PP} verbatim and using (\ref{dd2}) for $b$ (which, as we mentioned, is valid on $\RR^{2d}$), we obtain the following estimate: for every $h>0$ there exists $C>0$
(resp. there exist $h,C>0$) such that
\beqs
\left|D^{\gamma}_w
p'_{k,j}(w)\right|\leq C h^{|\gamma|+2k}A_{|\gamma|+2k}b(w)\langle
w\rangle^{-\rho(|\gamma|+2k)},
\eeqs
for all $w\in\RR^{2d}$,
$\gamma\in\NN^{2d}$, $k,j\in\NN$ (recall that $p'_{k,j}=0$, for
$0\leq k<j$, $p'_{k,0}=0$ for $k\in\ZZ_+$, and $p'_{0,0}=b$). Now,
$a\sim \sum_j (-1)^j b_j$ with $b_j=R(\sum_k p'_{k,j})$, where
$R\geq1$ can be chosen to be the same for all $j\in\NN$ and the
following estimate holds: for every $h>0$ there exists $C>0$
(resp. there exist $h,C>0$) such that
\beq\label{est_anti-wick}
\left|D^{\gamma}_wb_j(w)\right|\leq C
h^{|\gamma|+2j}A_{|\gamma|+2j}b(w)\langle
w\rangle^{-\rho(|\gamma|+2j)},
\eeq
for all $w\in\RR^{2d}$,
$\gamma\in\NN^{2d}$, $j\in\NN$ (cf. \cite[Lemma 3.1]{PP} and its
proof). Clearly $b_0=p'_{0,0}=b$. In the $(M_p)$ case, fix
$0<h'<1$ and let $C'>1$ be the constant for which
(\ref{est_anti-wick}) holds and in the $\{M_p\}$ case, let
$h',C'>1$ be the constants for which this estimate holds. If we
take large enough $R'$ such that $R'^{\rho}\geq 4c_0^2HLC'$ in the
$(M_p)$ case and $R'^{\rho}\geq 4c_0^2h'HLC'$ in the $\{M_p\}$
case respectively, then $a'=R'(\sum_j(-1)^j
b_j)\in\Gamma_{A_p,\rho}^{*,\infty}(\RR^{2d})$ is real-valued and
$a'\sim a$, i.e. $a-a'\in\SSS^*(\RR^{2d})$ (cf. Propositions \ref{subordinate} and \ref{eqsse}). Moreover, since $1-\chi_{j,R'}=0$ on $Q_{R'm_j}$ and $m_j^{2j}\geq M_{2j}/(c_0H^{2j})$, $\forall j\in\ZZ_+$,
\begin{align*}
\sum_{j=1}^{\infty}(1-\chi_{j,R'}(w))|b_j(w)|
&
\leq
C'b(w)\sum_{j=1}^{\infty}(1-\chi_{j,R'}(w))h'^{2j}A_{2j}\langle
w\rangle^{-2j\rho}
\\
&
\leq
C'b(w)\sum_{j=1}^{\infty}h'^{2j}A_{2j}R'^{-2j\rho}m_j^{-2j\rho}
\\
&
\leq c_0^2C'b(w)\sum_{j=1}^{\infty}(h'HL/R'^\rho)^{2j}\leq
b(w)/3.
\end{align*}
Thus
\beqs
a'(w)=b(w)+\sum_{j=1}^{\infty}(-1)^j(1-\chi_{j,R'}(w))b_j(w)\geq
2b(w)/3>0,\quad \forall w\in\RR^{2d}.
\eeqs
Hence $(A_{a'}\varphi,\varphi)\geq0$, $\varphi\in\SSS^*(\RR^d)$ (cf. \cite[Proposition 3.4]{PP}). Observe that $A_{a'}=b^w+T'$, for some $*$-regularising operator $T'$. Since $b$ is real-valued, $(b^w\varphi,\varphi)\in\RR$, $\varphi\in\SSS^*(\RR^d)$, hence the same holds for $T'$ too. We conclude $(b^w\varphi,\varphi)\geq -(T'\varphi,\varphi)\geq
-\|T'\|_{\mathcal{L}_b(L^2(\RR^d))}\|\varphi\|^2_{L^2(\RR^d)}$.
\end{proof}

Using Proposition \ref{maximalreal}, Proposition
\ref{semi-bound} and Remark \ref{ktv957939}, we can prove the following
spectral result in the same way as in the proof of \cite[Theorem 4.2.9, p. 163]{NR}.

\begin{proposition}\label{discretness_of_spe}
Let $a\in\Gamma_{A_p,\rho}^{*,\infty}(\RR^{2d})$ be a hypoelliptic
real-valued symbol such that $|a(w)|\rightarrow \infty$ as
$|w|\rightarrow \infty$ and let $A$ be the unbounded
operator on $L^2(\RR^d)$ defined by $a^w$. Then the closure $\overline{A}$ of $A$
is a self-adjoint operator having spectrum given by a sequence of
real eigenvalues either diverging to $+\infty$ or to $-\infty$
according to the sign of $a$ at infinity. The eigenvalues have
finite multiplicities and the eigenfunctions belong to
$\SSS^*(\RR^d)$. Moreover, $L^2(\RR^d)$ has an orthonormal basis
consisting of eigenfunctions of $\overline{A}$.
\end{proposition}

\section{The Weyl asymptotic formula for infinite order $\Psi$DOs. Part I: statements of the main results} \label{Section Weyl formulae, part I}

This section is dedicated to Weyl asymptotic formulae for a large class of infinite order hypoelliptic pseudo-differential operators. We state here our main results, their proofs are postponed to Section \ref{proofWeylasymp}, after obtaining some auxiliary results on the spectrum of the heat parametrix of positive hypoelliptic symbols.

We consider throughout this section a real-valued hypoelliptic symbol $a\in\Gamma_{A_p,\rho}^{*,\infty}(\RR^{2d})$ such that $a(w)\to\infty$ as $|w|\to\infty$. If we denote as $\overline{A}$ the closure of the unbounded operator on $L^2(\RR^d)$ induced by its Weyl quantisation $a^w$
then we can apply Proposition \ref{discretness_of_spe} to obtain
that the spectrum of the self-adjoint operator  $\overline{A}$ is
given by a sequence of real eigenvalues with finite multiplicities $\{\lambda_j\}_{j\in\mathbb{N}}$ which tends
to $\infty$, where multiplicities are taken into account and the sequence is arranged in non-decreasing order $\lambda_{0}\leq \lambda_{1}\leq \lambda_{2}\leq \dots\leq \lambda_{j}\leq \dots$. We denote the spectral counting function of the operator $A=a^{w}$ as
$$
N(\lambda)=\sum_{\lambda_j\leq \lambda}1=\#\{j\in\mathbb{N}|\,\lambda_{j}\leq \lambda\}.
$$

Our goal is to show later the following three theorems on spectral asymptotics. For these results, we will suppose that the symbol $a$ satisfies certain asymptotic bounds with respect to a comparison function $f$, which we assume throughout the rest of this section to be positive, strictly increasing, of ultrapolynomial growth of class $*$ on some interval $[Y,\infty)$, for some $Y>0$, and absolutely continuous on each compact subinterval of $[Y,\infty)$. Furthermore, we employ the notation
\begin{equation}
\label{defsigma}
\sigma(\lambda)=(f^{-1}(\lambda))^{2d} \quad \mbox{ for large } \lambda.
\end{equation}
\begin{theorem}
\label{Weylth1} Let $a\in\Gamma_{A_p,\rho}^{*,\infty}(\RR^{2d})$ hypoelliptic, let
 $f$ satisfy
\beq
\label{weyleq2}
\lim_{y\to\infty} \frac{yf'(y)}{f(y)}=\infty,
\eeq
and let $\Phi$ be a positive continuous function on the sphere $\mathbb{S}^{2d-1}$.
Suppose that for each $\varepsilon\in (0,1)$ there are positive constants $c_{\epsilon},C_{\epsilon},B_{\epsilon}>0$ such that
\beq
\label{weyleq3}
c_{\varepsilon}f((1-\varepsilon) r \Phi(\vartheta))\leq a(r\vartheta)\leq C_{\varepsilon}f((1+\varepsilon) r \Phi(\vartheta)),
\eeq
for all $r\geq B_{\varepsilon}$ and $\vartheta\in\mathbb{S}^{2d-1}$. Then,
\beq
\label{weyleq4}
\lim_{\lambda\to\infty}\frac{N(\lambda)}{\sigma(\lambda)}= \frac{\pi}{(2\pi)^{d+1}d}\int_{\mathbb{S}^{2d-1}}\frac{d\vartheta}{(\Phi(\vartheta))^{2d}}\:,
\eeq
\beq
\label{weyleq4.1}
\lambda_{j}=f\left(\gamma j^{\frac{1}{2d}}(1+o(1))\right) , \quad j\to\infty,
\eeq
with $\gamma=\sqrt{2\pi}\cdot (2d/\int_{\mathbb{S}^{2d-1}}(\Phi(\vartheta))^{-2d}d\vartheta)^{\frac{1}{2d}}$,
and, for each $h'<\gamma<h$,
\beq
\label{weyleq4.2}
\lim_{j\to\infty}\frac{\lambda_{j}}{f(h' j^{\frac{1}{2d}})}=\infty
\quad \mbox{and} \quad \lim_{j\to\infty}\frac{\lambda_{j}}{f(h j^{\frac{1}{2d}})}=0.
\eeq
\end{theorem}

Note that Theorem \ref{Weylth1} deals with operators which are truly of infinite order because integration of (\ref{weyleq2}) gives that $\langle w \rangle^{\beta}=o(a(w))$ for any $\beta>0$.

The next theorem gives the Weyl asymptotic formula for a wider class of finite order pseudo-differential operators than the one that is usually discussed in the literature, see e.g. \cite[Sect. 4.6]{NR}; in particular, our result is more general than \cite[Theorem 4.6.1, p. 196]{NR} (see Example \ref{Weylex2} below). The reader should also compare this with \cite[Theorem 30.1, p. 224]{Shubin}; we work with different assumptions than in the quoted result and, on the other hand, we give a more explicit result concerning the asymptotic behaviour of $N(\lambda)$.

\begin{theorem}
\label{Weylth2} Let $a\in\Gamma_{\rho}^m(\RR^{2d})$ be hypoelliptic (in the $\Gamma_{\rho}^m$-sense). Suppose that
\beq
\label{weyleq5}
\lim_{y\to\infty} \frac{yf'(y)}{f(y)}=\beta\in (0,\infty)
\eeq
exists. If
\beq
\label{weyleq6}
\lim_{r\to\infty} \frac{a(r\vartheta)}{f(r)}=\Phi(\vartheta)>0
\eeq
exists uniformly on $\vartheta\in\mathbb{S}^{2d-1}$, then
\beq
\label{weyleq7}
\lim_{\lambda\to\infty}\frac{N(\lambda)}{\sigma(\lambda)}= \frac{\pi}{(2\pi)^{d+1}d} \int_{\mathbb{S}^{2d-1}}\frac{d\vartheta}{(\Phi(\vartheta))^{2d/\beta}}
\eeq
and
\beq
\label{weyleq7.1}
\lambda_{j}\sim \left(\frac{\pi}{(2\pi)^{d+1}d} \int_{\mathbb{S}^{2d-1}}\frac{d\vartheta} {(\Phi(\vartheta))^{2d/\beta}}\right)^{-\frac{\beta}{2d}} f(j^{\frac{1}{2d}}) , \quad j\to\infty.
\eeq
\end{theorem}

We will derive the following ``geometric'' version of Theorems \ref{Weylth1} and \ref{Weylth2} where the asymptotic behaviour of $N$ is given in terms of the symbol.

\begin{corollary}
\label{Weylformulac1} Suppose that the symbol $a$ satisfies either the assumptions of Theorem \ref{Weylth1} or those of Theorem \ref{Weylth2}. Then,
\beq
\label{weyleqextrac1}
N(\lambda)\sim \frac{1}{(2\pi)^{d}} \int_{a(w) <\lambda}dw,\quad \lambda\to\infty.
\eeq
\end{corollary}
If one is only interested in upper $O$-estimates on $N$, the next theorem gives such bounds under much weaker assumptions on the symbol.

\begin{theorem}
\label{Weylth3} Let $a\in\Gamma_{A_p,\rho}^{*,\infty}(\RR^{2d})$ be hypoelliptic such that
\beq
\label{weyleq9}
Cf(|w|)\leq a(w) \quad \mbox{for all }|w|\geq B,
\eeq
for some $C,B>0$. If $f$ satisfies
\beq
\label{weyleq8}
0<\beta'=\liminf_{y\to\infty} \frac{yf'(y)}{f(y)},
\eeq
then,
\beq
\label{weyleq10}
\limsup_{\lambda \to\infty}\frac{N(\lambda)}{\sigma(\lambda)}\leq \frac{e}{2^d d!}\left(1+\frac{\Gamma(1+2d/\beta')}{C^{2d/\beta'}}\right)
\eeq
and for each $0<h< \sqrt{2}C^{1/\beta'}e^{-1/(2d)} d!^{1/(2d)}(C^{2d/\beta'}+\Gamma(1+2d/\beta'))^{-1/(2d)}$
\beq
\label{weyleq10.1}
\lambda_{j} \geq f(h j^{\frac{1}{2d}}), \quad j\geq j_h.
\eeq
Furthermore, if $f$ satisfies
\beq\label{weyleq811}
\lim_{y\to\infty} \frac{yf'(y)}{f(y)}=\beta'\in (0,\infty],
\eeq
then,
\beq\label{weyleq1011}
\limsup_{\lambda\rightarrow\infty}\frac{N(\lambda)}{\sigma(\lambda)}\leq \frac{\Gamma(1+2d/\beta')\:e}{2^{d} C^{2d/\beta'}d!} \quad \left(=\frac{e}{2^{d}d!} \quad \mbox{ if }\beta'=\infty\right)
\eeq
and the bound (\ref{weyleq10.1}) holds for each $0<h<\sqrt{2}C^{1/\beta'} d!^{1/(2d)}(e\,\Gamma(1+2d/\beta'))^{-1/(2d)}$ ($=\sqrt{2}(d!/e)^{1/(2d)}$ if $\beta'=\infty$).
\end{theorem}

\begin{remark} If $\limsup_{y\to\infty} yf'(y)/f(y)<\infty$, Theorem \ref{Weylth3} is also valid for $a\in\Gamma_{\rho}^m(\RR^{2d})$ that is $\Gamma_{\rho}^m$-hypoelliptic and satisfies (\ref{weyleq9}), as the proof given in Section \ref{proofWeylasymp} shows. Here we get that $\lambda_{j}$ is bounded from below by a constant multiple of $f(j^{\frac{1}{2d}})$ for $\lambda_j> 0$. In particular, this case applies to $f(y)=y^{\beta'}$, where we obtain $N(\lambda)=O(\lambda^{2d/\beta'})$ and $\lambda_{j}\geq h^{\beta'} j^{\beta'/(2d)}$, $j\geq j_{h}$, with the constants as in Theorem \ref{Weylth3} (see also Example \ref{Weylex2}).
\end{remark}

The rest of this section is devoted to some illustrative examples. The asymptotic formulae from Examples \ref{Weylex1} and \ref{Weylex3} prove a result that one might expect: the eigenvalues of a truly infinite order operator are ``very sparse''.

\begin{example} \label{Weylex1} If $f(y)=e^{(hy)^{1/s}}$ where $s>1$, then $\sigma(\lambda)\sim h^{-2d}(\ln\lambda)^{2ds}$ and, when $\Phi(\vartheta)=1$ Theorem \ref{Weylth1} delivers
\beq\label{formulaforhypop}
N(\lambda)\sim 2^{-d}h^{-2d}d!^{-1} (\ln\lambda)^{2ds}, \quad \lambda \to\infty,
\eeq
and
\beq\label{formulaforhypop2}
\lambda_{j}=\exp\left( 2^{1/(2s)}h^{1/s}d!^{1/(2ds)}j^{1/(2ds)}\left( 1+o(1)\right)\right),  \quad j \to\infty,
\eeq
because here  $\gamma=(d!)^{1/(2d)}\sqrt{2}$.

Let us give an example of a symbol that satisfies the assumptions in Theorem \ref{Weylth1} with this $f$. Let
$$a(w)=e^{(h\langle w\rangle)^{1/s}}+a_1(w),$$
where $s\geq 1/(1-\rho)$ is such $e^{\langle w\rangle^{1/s}}$ is of ultrapolynomial growth of class $*$ (i.e. $M_{p}\subset p!^{s}$ and $M_{p}\prec p!^{s}$, respectively) and $a_1$ is real-valued and satisfies the following estimate: for every $h'>0$ there exists $C'>0$ (resp. there exist $h',C'>0$) such that
\beq\label{kvr995551}
|D^{\alpha}_w a_1(w)|\leq C'h'^{|\alpha|}A_{\alpha}e^{(h\langle w\rangle)^{1/s}}\langle w\rangle^{-\rho(|\alpha|+1)},\,\, \forall w\in\RR^{2d},\, \forall \alpha\in\NN^{2d}.
\eeq
Clearly $a$ satisfies the bound
\beq\label{kts951317}
C_1e^{(h|w|)^{1/s}}\leq a(w)\leq C_2e^{(h|w|)^{1/s}},\,\, \mbox{for large}\,\, |w|.
\eeq
Furthermore, since $|D^{\alpha}_w\langle w\rangle|\leq 2^{|\alpha|+1}|\alpha|!\langle w\rangle^{1-|\alpha|}$, for all $w\in\RR^{2d}$, $\alpha\in\NN^{2d}$, \cite[Remark 7.6]{PP1} proves that $e^{(h\langle w\rangle)^{1/s}}\in \Gamma_{A_p,\rho}^{*,\infty}(\RR^{2d})$ and it is hypoelliptic. Because of (\ref{kvr995551}) and (\ref{kts951317}), $a$ is also a hypoelliptic symbol in $\Gamma_{A_p,\rho}^{*,\infty}(\RR^{2d})$. Hence, the asymptotic formulae (\ref{formulaforhypop}) and (\ref{formulaforhypop2}) for $N(\lambda)$ and the eigenvalues hold true for $a^w=(e^{(h\langle \cdot\rangle)^{1/s}})^w+a_1^w$. We remark that given any $s>1$ the conditions are always met with $\nu/l\leq\rho\leq 1-1/s$, $M_{p}= p!^{l}$, and $A_{p}=p!^{\nu}$ if we choose the parameters $l$ and $\nu$ such that $1<\nu<l<s$ and $\nu/l\leq 1-1/s$.

\indent More generally, let $f(y)=\tilde{M}(hy)$, where $\tilde{M}$ is the associated function of a sequence $M_{p}\subset \tilde{M}_{p}$ (resp. $M_{p}\prec \tilde{M}_{p} $), and $\tilde{M}_p$ satisfies $(M.1)$. Then \cite{Komatsu1} $yf'(y)/f(y)=\tilde{m}(hy)\to\infty$. In this case, when $\Phi(\vartheta)=1$ we obtain
\beq\label{asymforophyp}
N(\lambda)\sim 2^{-d}h^{-2d}d!^{-1} (\tilde{M}^{-1}(\ln \lambda))^{2d} , \quad \lambda \to\infty.
\eeq
Similarly for the upper bound from Theorem \ref{Weylth3}. In particular, if there exist $C,h>0$ such that $Ce^{\tilde{M}(h|w|)}\leq a(w)$, for large $|w|$, one always has the $O$-bound
$$
N(\lambda)= O((\tilde{M}^{-1}(\ln \lambda))^{2d}).
$$
If $M_p\prec \tilde{M}_p$ and there exists $B>0$ such that for every $h>0$ there exists $c>0$ such that $ce^{\tilde{M}(h|w|)}\leq a(w)$, $\forall|w|\geq B$, then we have the effective estimate
$$
N(\lambda)\leq 2(2h^{2})^{-d}(e/d!) (\tilde{M}^{-1}(\ln \lambda))^{2d}
$$ for large enough $\lambda\geq \lambda_h$, which yields the
 $o$-bound
\[
N(\lambda)= o((\tilde{M}^{-1}(\ln \lambda))^{2d}), \quad \lambda\to\infty.
\]
\end{example}

\begin{example}
\label{Weylex3} We present in this example another nontrivial instance of a hypoelliptic pseudo-differential operator of infinite order. Let $\nu, l, s$ be three positive numbers such that $1<\nu<l<s$ and $\nu/l\leq 1-1/s$. Consider the entire function
$$
P(z)=1+\sum_{n=1}^{\infty} \frac{(hz)^{n}}{n^{sn}}, \quad z\in\mathbb{C},
$$
where $h$ is a positive constant, and the symbol
$$
a(w)=P(\langle w\rangle), \quad w\in\mathbb{R}^{2d}.
$$
It is shown in \cite[Sect. 3] {CPP1} that  $a\in\Gamma_{A_p,\rho}^{*,\infty}(\RR^{2d})$ is hypoelliptic, where $\nu/l\leq\rho\leq 1-1/s$, $M_{p}= p!^{l}$, and $A_{p}=p!^{\nu}$. Denote as $N$ the spectral counting function of the Weyl quantisation of $a$ and let $\{\lambda_{j}\}_{j=0}^{\infty}$ be its sequence of eigenvalues. We will show that
\begin{equation}
\label{eqexwa2}
N(\lambda)\sim \frac{e^{2ds}}{2^dh^{2d}s^{2ds}d!} (\ln \lambda)^{2d s}
\end{equation}
and
\begin{equation}
\label{eqexwa3}
\lambda_{j}=\exp\left(e^{-1}s\cdot 2^{1/(2s)}h^{1/s}d!^{1/(2ds)}j^{1/(2ds)}\left( 1+o(1)\right)\right), \quad j\to\infty.
\end{equation}

We start by noticing that, given any fixed $0<\varepsilon<1$, we have bounds
$$
c'_{\varepsilon} P((1-\varepsilon)|w|) \leq a(w)\leq C'_{\varepsilon} P((1+\varepsilon)|w|)
$$
for sufficiently large $w$. Next, observe that
\begin{equation}
\label{Gelfand-Shilovinequality}
e^{-s} \exp\left(\frac{sy^{1/s}}{e}\right) \leq \sup_{p\in \mathbb{Z}_{+}} \frac{y^{p}}{p^{sp}}\leq e^s \exp\left(\frac{sy^{1/s}}{e}\right), \quad y\geq e^{s},
\end{equation}
because the only critical point of $g(t)=t\ln y - st \ln t$ lies at $t= e^{-1}y^{1/s}$. Thus, given any arbitrary $0<\varepsilon<1$, we obtain the bounds
$$
e^{-s} \exp\left(\frac{s(hy)^{1/s}}{e}\right) \leq P(y)\leq \frac{((1+\varepsilon)e)^{s}}{(\varepsilon+1)^{s}-1} \exp\left(\frac{(1+\varepsilon)s(hy)^{1/s}}{e}\right), \quad y\geq e^{s}/h.
$$
It then follows that the radial symbol $a$ satisfies (\ref{weyleq3}) with $f(y)=\exp(e^{-1}s(hy)^{1/s})$ and the constant function $\Phi(\vartheta)=1$. Theorem \ref{Weylth1} immediately yields  \eqref{eqexwa2} and \eqref{eqexwa3}.
\end{example}

\begin{example}
\label{Weylex2} If $f(y)=y^{\beta}\ln^{\alpha}y$, where $\beta>0$, we have that $yf'(y)/f(y)\to\beta$ and $\sigma(\lambda)\sim (\beta^{\alpha} \lambda\ln^{\alpha} \lambda)^{1/\beta} $. Therefore, the conclusion of Theorem \ref{Weylth2} reads in this case
$$N(\lambda)\sim \frac{(\beta^{\alpha}\lambda)^{2d/\beta}\pi}{(2\pi)^{d+1}d\ln^{2d\alpha/\beta}\lambda} \int_{\mathbb{S}^{2d-1}}\frac{d\vartheta}{(\Phi(\vartheta))^{2d/\beta}}, \quad \lambda \to\infty,$$
and
$$
\lambda_{j}\sim (2d)^{\frac{\beta-\alpha}{2d}}(2\pi)^{\frac{\beta}{2}}\left(\int_{\mathbb{S}^{2d-1}}\frac{d\vartheta}{(\Phi(\vartheta))^{2d/\beta}}\right)^{-\frac{\beta}{2d}} j^{\frac{\beta}{2d}}\ln^{\frac{\alpha}{2d}}j , \quad j\to\infty.
$$
Likewise for the upper bound from Theorem \ref{Weylth3}.
\end{example}

\section{The spectrum of the heat parametrix}\label{Section heat parametrix}

Throughout this section we assume $a$ is a hypoelliptic real-valued symbol in $\Gamma^{*,\infty}_{A_p,\rho}(\RR^{2d})$ such that $a(w)/\ln|w|\to \infty$ as $|w|\to\infty$. There exists $B\geq1$ such that the hypoellipticity condition (\ref{dd2}) for $a$ holds on $Q^c_B$ and $a(w)>0$, $\forall w\in Q^c_B$. Pick $\tilde{\chi}\in\DD^{(A_p)}(\RR^{2d})$ (resp. $\tilde{\chi}\in\DD^{\{A_p\}}(\RR^{2d})$) such that $0\leq \tilde{\chi}\leq 1$, $\tilde{\chi}=1$ on $Q_{B_1}$, for $B_1>B$, and $\tilde{\chi}=0$ on the complement of a small neighbourhood of $\overline{Q_{B_1}}$. Then $b=(1-\tilde{\chi})a+\tilde{\chi}$ is positive on the whole $\RR^{2d}$ and, in fact, it is a hypoelliptic symbol in $\Gamma^{*,\infty}_{A_p,\rho}(\RR^{2d})$ for which the hypoellipticity condition (\ref{dd2}) holds globally on $\RR^{2d}$.

\subsection{The heat parametrix of positive hypoelliptic symbols}\label{sub heat parametrix}

For the symbol $b$ constructed above, we can apply the theory given in \cite[Subsection 7.2]{PP1} for the construction of the heat parametrix. We have the following series of results.\\
\indent There exist $u_j(t,w)\in C^{\infty}(\RR\times\RR^{2d})$, $j\in\NN$, such that $u_0(t,w)=e^{-tb(w)}$ and the following results hold.

\begin{lemma}\label{est_heat_ker}$($\cite[Lemma 7.8]{PP1}$)$
For every $h>0$ there exists $C>0$ (resp. there exist $h,C>0$) such that
\beqs
|D^n_tD^{\alpha}_wu_j(t,w)|\leq Cn!h^{|\alpha|+2j}A_{|\alpha|+2j}\left(b(w)\right)^n\langle w\rangle^{-\rho(|\alpha|+2j)}e^{-\frac{t}{4}b(w)},
\eeqs
for all $\alpha\in\NN^{2d}$, $n\in\NN$, $(t,w)\in[0,\infty)\times\RR^{2d}$.
\end{lemma}

Notice that for each $R>0$, the function $u(t,w)=\sum_{n=0}^{\infty} (1-\chi_{n,R}(w))u_n(t,w)=R(\sum_j u_j)(t,w)$ is in $ C^{\infty}(\RR\times\RR^{2d})$.

\begin{lemma}\label{rks75}$($\cite[Lemma 7.10]{PP1}$)$
There exists $R>1$ such that the $C^{\infty}$-function $u(t,w)=\sum_{n=0}^{\infty} (1-\chi_{n,R}(w))u_n(t,w)=R(\sum_j u_j)(t,w)$ satisfies the following condition: for every $h>0$ there exists $C>0$ (resp. there exist $h,C>0$) such that
\beq\label{est_heat_par}
|D^n_tD^{\alpha}_wu(t,w)|\leq Cn!h^{|\alpha|}A_{\alpha}\left(b(w)\right)^n\langle w\rangle^{-\rho|\alpha|}e^{-\frac{t}{4}b(w)},
\eeq
for all $\alpha\in\NN^{2d}$, $n\in\NN$, $(t,w)\in[0,\infty)\times\RR^{2d}$ and
\beqs
\sup_{k\in\ZZ_+}\sup_{\substack{\alpha\in\NN^{2d}\\ n\in\NN}}\sup_{\substack{w\in Q^c_{3Rm_k}\\ t\in [0,\infty)}}\frac{\left|D^n_tD^{\alpha}_w\left(u(t,w)-\sum_{j<k} u_j(t,w)\right)\right|\langle w\rangle^{\rho(|\alpha|+2k)}}{n!h^{|\alpha|+2k}A_{|\alpha|+2k}\left(b(w)\right)^n e^{-\frac{t}{4}b(w)}}\leq C.
\eeqs
\end{lemma}

\begin{theorem}\label{sol_heat_ker}$($\cite[Theorem 7.11]{PP1}$)$
The function $u(t,w)$ of Lemma \ref{rks75} defines the vector-valued mapping $\mathbf{u}:t\mapsto u(t,\cdot)$, $[0,\infty)\rightarrow \Gamma_{A_p,\rho}^{*,\infty}(\RR^{2d})$, that belongs to $C^{\infty}([0,\infty); \Gamma_{A_p,\rho}^{*,\infty}(\RR^{2d}))$. The operator-valued mapping $t\mapsto (\mathbf{u}(t))^w$ belongs to both
$ C^{\infty}([0,\infty);\mathcal{L}_b(\SSS^*(\RR^d),\SSS^*(\RR^d)))$ and $ C^{\infty}([0,\infty);\mathcal{L}_b(\SSS'^*(\RR^d),\SSS'^*(\RR^d)))$. Moreover, $(\mathbf{u}(t))^w$ satisfies
\beq\label{system2}
\left\{\begin{array}{l}
(\partial_t+b^w)(\mathbf{u}(t))^w=\mathbf{K}(t),\,t\in[0,\infty),\\
(\mathbf{u}(0))^w=\mathrm{Id},\\
\end{array}\right.
\eeq
where $\mathbf{K}\in  C^{\infty}([0,\infty);\mathcal{L}_b(\SSS'^*(\RR^d),\SSS^*(\RR^d)))$.\\
\indent For each $t\geq 0$, $(\mathbf{u}(t))^w\in\mathcal{L}(L^2(\RR^d))$ and there exists $C>0$ such that
 $$\|(\mathbf{u}(t))^w\|_{\mathcal{L}_b(L^2(\RR^d))}\leq C, \quad \mbox{for all } t\geq 0.$$
  The mapping $t\mapsto (\mathbf{u}(t))^w$, $(0,\infty)\rightarrow \mathcal{L}_b(L^2(\RR^d))$, is continuous and $(\mathbf{u}(t))^w\rightarrow (\mathbf{u}(0))^w=\mathrm{Id}$, as $t\rightarrow0^+$, in $\mathcal{L}_p(L^2(\RR^d))$. Furthermore, for each $n\in\ZZ_+$ and $t>0,$ $(\partial^n_t\mathbf{u}(t))^w\in\mathcal{L}(L^2(\RR^d))$. The mapping $t\mapsto (\mathbf{u}(t))^w$, $(0,\infty)\rightarrow \mathcal{L}_b(L^2(\RR^d))$, is smooth and $\partial^n_t(\mathbf{u}(t))^w=(\partial^n_t\mathbf{u}(t))^w$.
\end{theorem}

Since the operator $a^w-b^w=(a-b)^w$ is $*$-regularising (by the definition of $b$), (\ref{system2}) implies
\beq\label{system3}
\left\{\begin{array}{l}
(\partial_t+a^w)(\mathbf{u}(t))^w=\tilde{\mathbf{K}}(t),\,t\in[0,\infty),\\
(\mathbf{u}(0))^w=\mathrm{Id},\\
\end{array}\right.
\eeq
where $\tilde{\mathbf{K}}\in  C^{\infty}([0,\infty);\mathcal{L}_b(\SSS'^*(\RR^d),\SSS^*(\RR^d)))$.\\
\indent We denote by $A$ the unbounded operator on $L^2(\RR^d)$ defined by $a^w$. We apply Proposition \ref{discretness_of_spe} and obtain
that the spectrum of the self-adjoint operator $\overline{A}$ is
given by a sequence of real eigenvalues $\{\lambda_j\}_{j\in\mathbb{N}}$ which tends
to $+\infty$, where the multiplicities are taken into account, and $L^2(\RR^d)$ has an orthonormal basis
$\{\varphi_j\}_{j\in\NN}$ consisting of eigenfunctions of $\overline{A}$ which
all belong to $\SSS^*(\RR^d)$ ($\varphi_j$ corresponds to
$\lambda_j$, $j\in\NN$). For each $t\geq0$, we define the following
operator on $L^2(\RR^d)$
\beq\label{semigroup11}
T(t)g=\sum_{j=0}^{\infty}
e^{-t\lambda_j}(g,\varphi_j)\varphi_j,\,\, g\in L^2(\RR^d).
\eeq
Obviously, the above series is unconditionally convergent and $T(t)$
is continuous. Furthermore, $T(t)$ is self-adjoint (one easily
verifies that $(T(t)g,g)\in [0,\infty)$, $g\in L^2(\RR^d)$, and hence it
is positive) and $T(0)=\mathrm{Id}$. Clearly, $\{T(t)\}_{t\geq 0}$ is a
$C_0$-semigroup.\\
\indent As it will become clear later, the analysis of this semigroup is one of the key ingredients in the proofs of the main results from Section \ref{Section Weyl formulae, part I}. We will show:\\
\indent - $T(t)$ belongs to $\mathcal{L}(\SSS^*(\RR^d),\SSS^*(\RR^d))$;\\
\indent - the mapping $t\mapsto T(t)$, $[0,\infty)\rightarrow \mathcal{L}_b(\SSS^*(\RR^d),\SSS^*(\RR^d))$, is smooth;\\
\indent - $T(t)$ and $(\mathbf{u}(t))^w$ are the same, modulo a smooth $*$-regularising family.\\
As the proofs of these facts are rather lengthy, we devote a whole subsection to them.

\begin{remark}\label{heat-para-fin}
If $a\in \Gamma^m_{\rho}(\RR^{2d})$ is a hypoelliptic real-valued symbol such that $a(w)\geq c\langle w\rangle^{\delta}$ for some $\delta>0$, $\forall|w|\geq c$, one can construct its heat parametrix as well. For this purpose, one can use the same construction as in \cite[Theorem 4.5.1, p. 193]{NR} (although it is there given only for elliptic symbols). In fact, defining $b\in \Gamma^m_{\rho}(\RR^{2d})$ to be positive on $\RR^{2d}$ and equal to $a$ outside of a compact neighbourhood of the origin, one can repeat the proof of the quoted result verbatim to find a symbol $u(t,\cdot)\in\Gamma^m_{\rho}(\RR^{2d})$, $t\geq 0$, which solves (\ref{system2}) with $\mathbf{K}\in C^{\infty}([0,\infty);\mathcal{L}_b(\SSS'(\RR^d),\SSS(\RR^d)))$. Moreover, there are $u_j(t,w)\in C^{\infty}(\RR\times\RR^{2d})$, $j\in\NN$, such that
\beqs
t^k \left|D^n_tD^{\alpha}_w\left(u(t,w)-\sum_{j=0}^{J-1}u_j(t,w)\right)\right|\leq \frac{C_{k,n,J,t_0,\alpha}(b(w))^{n-k}}{\langle w\rangle^{\rho(|\alpha|+2J)}},\,\, w\in\RR^{2d}, t\in[0,t_0]
\eeqs
($t_0>0$ can be arbitrarily chosen), where $u_0(t,w)=e^{-tb(w)}$ and $u_j$ is given as  $u_j(t,w)=e^{-tb(w)}\sum_{l=1}^{2j}t^lu_{l,j}(w)$, $j\in\ZZ_+$, with symbols $u_{l,j}$ that satisfy the estimates
\beqs
|D^{\alpha}_wu_{l,j}(w)|\leq C_{l,j,\alpha}(b(w))^l\langle w\rangle^{-\rho(|\alpha|+2j)},\,\, w\in\RR^{2d}.
\eeqs
Notice then that $(\mathbf{u}(t))^w=(u(t,\cdot))^w$ satisfies the equation (\ref{system3}) for some vector-valued function $\tilde{\mathbf{K}}\in  C^{\infty}([0,\infty);\mathcal{L}_b(\SSS'(\RR^d),\SSS(\RR^d)))$.
\end{remark}

\subsection{The analysis of the semigroup $T(t)$, $t\geq 0$}
\label{analysis heat semigroup}
\begin{lemma}
The infinitesimal generator of $\{T(t)\}_{t\geq 0}$ is
$-\overline{A}$.
\end{lemma}

\begin{proof} For the moment, denote as $B$ the infinitesimal
generator of $\{T(t)\}_{t\geq 0}$. Fix $\psi\in\SSS^*(\RR^d)$.
Since
$\overline{A}\psi=\sum_{j=0}^{\infty}(\overline{A}\psi,\varphi_j)\varphi_j$,
we have
$\sum_{j=0}^{\infty}|(\overline{A}\psi,\varphi_j)|^2<\infty$ and,
as $\overline{A}$ is self-adjoint, we conclude
\beqs
\sum_{j=0}^{\infty}\lambda_j^2|(\psi,\varphi_j)|^2<\infty\,\,
\mbox{and}\,\,
\overline{A}\psi=\sum_{j=0}^{\infty}\lambda_j(\psi,\varphi_j)\varphi_j,
\eeqs
where the last series is unconditionally convergent in
$L^2(\RR^d)$. We have
\beq\label{est11kk}
\frac{T(t)\psi-\psi}{t}+\overline{A}\psi=\sum_{j=0}^{\infty}\left(\frac{e^{-t\lambda_j}-1}t +\lambda_j\right)(\psi,\varphi_j)\varphi_j.
\eeq
Let $c>0$ be such that $\lambda_j> -c$, $j\in\NN$. By Taylor formula, there exists $C>0$ such that $|e^{-ts}-1|\leq
Ct|s|$, for all $t\in[0,1]$, $s\geq -c$. Hence
$|e^{-t\lambda_j}-1|\leq Ct|\lambda_j|$, for all $t\in[0,1]$,
$j\in\NN$. Thus, letting $t\rightarrow 0^+$ in (\ref{est11kk}),
dominated convergence implies
$t^{-1}(T(t)\psi-\psi)\rightarrow-\overline{A}\psi$ in
$L^2(\RR^d)$. Thus $-A\subset B$ and hence $-\overline{A}\subset
B$ ($B$ is closed as a generator of a $C_0$-semigroup). Now, for
$f,g\in D(B)$, we have
\beqs
(Bf,g)=\lim_{t\rightarrow0^+}(t^{-1}((T(t)f-f),g)=\lim_{t\rightarrow0^+}(f,t^{-1}(T(t)g-g)) =(f,Bg),
\eeqs
i.e. $B\subset B^*$. Since $B^*\subset
-\overline{A}^*=-\overline{A}$ (which follows from
$-\overline{A}\subset B$), we conclude $-\overline{A}=B$.
\end{proof}

Let $c>0$ be large enough such that $\lambda_j>-c+1$, $j\in\NN$,
and $\tilde{a}(w)=a(w)+c>0$, $w\in\RR^{2d}$. Then
$\tilde{a}\in\Gamma_{A_p,\rho}^{*,\infty}(\RR^{2d})$ is
hypoelliptic and we denote by $\tilde{A}$ the corresponding
unbounded operator on $L^2(\RR^d)$. Notice that $\sigma(\overline{\tilde{A}})\subseteq \{\lambda\in\RR|\, \lambda> 1\}$ and $\overline{\tilde{A}}$ is self-adjoint (see Proposition \ref{maximalreal}).\\
\indent Denote by $\mathbf{P}$ the following closed sector:
$\{z\in\CC\backslash\{0\}|\, -3\pi/4\leq\arg z\leq 3\pi/4\}\cup\{0\}$. One easily
verifies that there exists $\tilde{C}>0$ such that
\beq\label{est_forss}
\tilde{a}(w)\leq
\tilde{C}|\tilde{a}(w)+z|\,\, \mbox{and}\,\, |z|\leq
\tilde{C}|\tilde{a}(w)+z|,\,\, \forall w\in\RR^{2d},\,\, \forall
z\in\mathbf{P}.
\eeq
Of course, $\tilde{a}(w)+z\neq 0$, for all
$w\in\RR^{2d}$, $z\in\mathbf{P}$. We denote by $\tilde{a}_z$ the
symbol $\tilde{a}+z\in\Gamma_{A_p,\rho}^{*,\infty}(\RR^{2d})$.
These inequalities yield that $\tilde{a}_z$, $z\in\mathbf{P}$, are
hypoelliptic and they satisfy the following uniform estimate: for
every $h>0$ there exists $C>0$ (resp. there exist $h,C>0$) such
that
\beq\label{est_for_ap}
\left|D^{\alpha}\tilde{a}_z(w)\right|\leq C
h^{|\alpha|}A_{\alpha}|\tilde{a}_z(w)|\langle
w\rangle^{-\rho|\alpha|},\,\, w\in\RR^{2d},\, \alpha\in\NN^{2d},\,
z\in\mathbf{P}.
\eeq
Notice that (\ref{est_forss}) implies that
there exist $c,C,m>0$ (resp. for every $m>0$ there exist $c,C>0$)
such that
\beq\label{est_bee}
c(1+|z|)e^{-M(m|\xi|)}e^{-M(m|x|)}\leq |\tilde{a}_z(x,\xi)|\leq
C(1+|z|)e^{M(m|\xi|)}e^{M(m|x|)},
\eeq
for all
$(x,\xi)\in\RR^{2d}$, $z\in\mathbf{P}$. In the Roumieu case, employing Lemma \ref{lemulgr117}, this estimate yields the existence of $(k_p)\in\mathfrak{R}$ and $c,C>0$ such that
\beq\label{est_rou}
c(1+|z|)e^{-N_{k_p}(|\xi|)}e^{-N_{k_p}(|x|)}\leq
|\tilde{a}_z(x,\xi)|\leq
C(1+|z|)e^{N_{k_p}(|\xi|)}e^{N_{k_p}(|x|)},
\eeq
for all
$(x,\xi)\in\RR^{2d}$, $z\in\mathbf{P}$. Define
$q^{(z)}_0(w)=1/\tilde{a}_z(w)$, $w\in\RR^{2d}$, and inductively
\beqs
q^{(z)}_j(x,\xi)=-q^{(z)}_0(x,\xi)\sum_{s=1}^j\sum_{|\alpha+\beta|=s}
\frac{(-1)^{|\beta|}}{\alpha!\beta!2^s}\partial^{\alpha}_{\xi}
D^{\beta}_x q^{(z)}_{j-s}(x,\xi) \partial^{\beta}_{\xi}
D^{\alpha}_x \tilde{a}_z(x,\xi),\,\, (x,\xi)\in \RR^{2d}.
\eeqs
In a completely analogous way as in \cite[Subsection 6.2.1]{PP1},
one proves that $\sum_j q^{(z)}_j\in
FS_{A_p,\rho}^{*,\infty}(\RR^{2d};0)$, $\sum_j q^{(z)}_j\#
\tilde{a}_z=\mathbf{1}=\tilde{a}_z\#\sum_j q^{(z)}_j$ in $FS_{A_p,\rho}^{*,\infty}(\RR^{2d};0)$ and the
following estimate holds: for every $h>0$ there exists $C>0$
(resp. there exist $h,C>0$) such that
\beq\label{estforthepar}
\left|D^{\alpha}_w q^{(z)}_j(w)\right|\leq
C\frac{h^{|\alpha|+2j}A_{|\alpha|+2j}}{|\tilde{a}_z(w)|\langle
w\rangle^{\rho(|\alpha|+2j)}},\,\, w\in
\RR^{2d},\,\alpha\in\NN^{2d},\, j\in\NN,\, z\in\mathbf{P}.
\eeq
This estimate together with (\ref{est_bee}) in the Beurling case
and (\ref{est_rou}) in the Roumieu case respectively, implies the
following:\\
\indent in the $(M_p)$ case, there exists $m>0$ such that for
every $h>0$ there is $C>0$ such that
\beq\label{estuniforminlamb}
(1+|z|)\left|D^{\alpha}_w q^{(z)}_j(w)\right|\leq
Ch^{|\alpha|+2j}A_{|\alpha|+2j}e^{M(m|\xi|)}e^{M(m|x|)}\langle
w\rangle^{-\rho(|\alpha|+2j)},
\eeq
for all $w\in \RR^{2d}$, $\alpha\in\NN^{2d}$, $j\in\NN$, $z\in\mathbf{P}$;\\
\indent in the $\{M_p\}$ case, there exist $(k_p)\in\mathfrak{R}$
and $h,C>0$ such that
\beq\label{estuniforminlamr}
(1+|z|)\left|D^{\alpha}_w q^{(z)}_j(w)\right|\leq
Ch^{|\alpha|+2j}A_{|\alpha|+2j}e^{N_{k_p}(|\xi|)}e^{N_{k_p}(|x|)}\langle
w\rangle^{-\rho(|\alpha|+2j)},
\eeq
for all $w\in \RR^{2d}$,
$\alpha\in\NN^{2d}$, $j\in\NN$, $z\in\mathbf{P}$. Thus, we have
obtained
\beq
\Big\{\ssum (1+|z|)q^{(z)}_j\big|\, z\in\mathbf{P}\Big\}&\precsim& e^{M(m|\xi|)}e^{M(m|x|)}\,\, \mbox{in}\,\, FS_{A_p,\rho}^{(M_p),\infty}(\RR^{2d};0)\,\, \mbox{and}\label{ths7975}\\
\Big\{\ssum (1+|z|)q^{(z)}_j\big|\, z\in\mathbf{P}\Big\}&\precsim&
e^{N_{k_p}(|\xi|)}e^{N_{k_p}(|x|)}\,\, \mbox{in}\,\,
FS_{A_p,\rho}^{\{M_p\},\infty}(\RR^{2d};0)\label{ths7977}
\eeq
in the Beurling and the Roumieu case, respectively. Similarly,
(\ref{est_bee}) and (\ref{est_for_ap}) yield
$\{\tilde{a}_{z}/(1+|z|)|\,z\in\mathbf{P}\}\precsim
e^{M(m|\xi|)}e^{M(m|x|)}$ in the Beurling case and (\ref{est_rou})
and (\ref{est_for_ap}) imply
$\{\tilde{a}_{z}/(1+|z|)|\,z\in\mathbf{P}\}\precsim
e^{N_{k_p}(|\xi|)}e^{N_{k_p}(|x|)}$ in the Roumieu case. Thus,
Corollary \ref{corweylqu} implies that there exist $R_1,R_2>0$
such that
$$\Big
\{\mathrm{Op}_{1/2}\big(R_1(\sum_j
q^{(z)}_j)\big)\tilde{a}_{z}^w-\mathrm{Id}\big|\,
z\in\mathbf{P}\Big\}
\quad \mbox{and} \quad
\Big\{\tilde{a}_{z}^w\mathrm{Op}_{1/2}\big(R_2(\sum_j
q^{(z)}_j)\big)-\mathrm{Id}\big|\, z\in\mathbf{P}\Big\}
$$
 are
equicontinuous subsets of
$\mathcal{L}(\SSS'^*(\RR^d),\SSS^*(\RR^d))$ (note that $R(\sum_j
(1+|z|)q^{(z)}_j)=(1+|z|)R(\sum_j q^{(z)}_j)$, for $R>0$). By
taking $R=\max\{R_1, R_2\}$, we obtain the next result (taking
larger $R_1$ or $R_2$ yields the same results because of
Proposition \ref{eqsse}).

\begin{proposition}\label{uniformparam}
There exists $R>0$, which can be taken arbitrary large, such that
\beqs
\Big\{\mathrm{Op}_{1/2}\big(R(\ssum
q^{(z)}_j)\big)\tilde{a}_{z}^w-\mathrm{Id}\big|\,
z\in\mathbf{P}\Big\}\quad \mbox{and}\quad
\Big\{\tilde{a}_{z}^w\mathrm{Op}_{1/2}\big(R(\ssum
q^{(z)}_j)\big)-\mathrm{Id}\big|\, z\in\mathbf{P}\Big\}
\eeqs
are
equicontinuous subsets of
$\mathcal{L}(\SSS'^*(\RR^d),\SSS^*(\RR^d))$. Moreover, the estimate
\eqref{estforthepar} holds for
$\{\sum_j q^{(z)}_j\}_{z\in\mathbf{P}}$.
\end{proposition}

\begin{lemma}\label{lemconofp}
There exists $R'>0$ such that for all $R\geq R'$ the following statements hold:
\begin{itemize}
\item[$(i)$] $q_z:=R(\sum_j q^{(z)}_j)\in \Gamma_{A_p,\rho}^{*,\infty}(\RR^{2d})$, $z\in\mathbf{P}$, and for every $h>0$ there exists $C>0$ (resp. there exist $h,C>0$) such that
\beq\label{estforthepar1}
\left|D^{\alpha}_w q_z(w)\right|\leq
\frac{Ch^{|\alpha|}A_{\alpha}}{|\tilde{a}_z(w)|\langle
w\rangle^{\rho|\alpha|}},\,\, w\in
\RR^{2d},\,\alpha\in\NN^{2d},\, z\in\mathbf{P};
\eeq
\item[$(ii)$] the set $\{(1+|z|)q^w_z|\, z\in\mathbf{P}\}$ is equicontinuous in both $\mathcal{L}(\SSS^*(\RR^d),\SSS^*(\RR^d))$ and $\mathcal{L}(\SSS'^*(\RR^d),\SSS'^*(\RR^d))$.
\end{itemize}
\end{lemma}

\begin{proof} The estimate (\ref{estforthepar}) implies $\{\sum_j q^{(z)}_j|\, z\in\mathbf{P}\}\precsim\{1/|\tilde{a}_z||\, z\in\mathbf{P}\}$ in $FS_{A_p,\rho}^{*,\infty}(\RR^{2d};0)$. Thus, we can apply Proposition \ref{subordinate} to obtain the existence of $R'>0$ such that for each $R\geq R'$, $q_z:=R(\sum_j q^{(z)}_j)\in \Gamma_{A_p,\rho}^{*,\infty}(\RR^{2d})$ and (\ref{estforthepar1}) is valid when $w\in Q^c_{Bm_1}=Q^c_B$, for some $B=B(R)>0$. There exists $j_0\in\ZZ_+$ such that $q_z(w)=\sum_{n=0}^{j_0}(1-\chi_{n,R}(w))q^{(z)}_n(w)$, for all $w\in Q_B$, $z\in\mathbf{P}$. Because of (\ref{estforthepar}) we can conclude the validity of (\ref{estforthepar1}) when $w\in Q_B$ as well, and the proof of $(i)$ is complete.\\
\indent Fix $R\geq R'$ and consider $q_z=R(\sum_j q^{(z)}_j)$, $z\in\mathbf{P}$. As a direct consequence of (\ref{estforthepar1}) and (\ref{est_bee}) (resp. (\ref{est_rou})), we have $\{(1+|z|)q_z|\, z\in\mathbf{P}\}\precsim e^{M(m|\xi|)}e^{M(m|x|)}$ (resp. $\{(1+|z|)q_z|\, z\in\mathbf{P}\}\precsim e^{N_{k_p}(|\xi|)}e^{N_{k_p}(|x|)}$). Hence, Proposition \ref{continuity} proves $(ii)$.
\end{proof}

Fix $R>0$ for which the conclusions in Proposition \ref{uniformparam} and Lemma \ref{lemconofp} hold and denote $q_z=R(\sum_jq^{(z)}_j)\in \Gamma_{A_p,\rho}^{*,\infty}(\RR^{2d})$, $z\in\mathbf{P}$. Since
$\sigma(\overline{\tilde{A}})\subseteq \{\lambda\in\RR|\, \lambda>
1\}$, it follows that $(z+\overline{\tilde{A}})$ is injective for each
$z\in\mathbf{P}$. Hence,  the operator
$\tilde{a}_z^w:\SSS^*(\RR^d)\rightarrow\SSS^*(\RR^d)$ is injective, as well. Moreover,
for given $\varphi\in\SSS^*(\RR^d)$, there exists $g\in
L^2(\RR^d)$ such that $(z+\overline{\tilde{A}})g=\varphi$ (as
$z\in\rho(\overline{\tilde{A}})$), i.e. $\tilde{a}_z^wg=\varphi$.
Since $\tilde{a}_z$ is hypoelliptic, it is globally regular and
hence $g\in\SSS^*(\RR^d)$. Thus $\tilde{a}_{z}^w$ is a continuous
bijection on $\SSS^*(\RR^d)$. As $\SSS^{(M_p)}(\RR^d)$ is an
$(F)$-space and $\SSS^{\{M_p\}}(\RR^d)$ is a $(DFS)$-space, it follows that
$\SSS^*(\RR^d)$ is a Pt\'{a}k space (see \cite[Sect. IV. 8, p.
162]{Sch}). The Pt\'{a}k homomorphism theorem \cite[Corollary 1, p.
164]{Sch} implies that $\tilde{a}_{z}^w$ is topological
isomorphism on $\SSS^*(\RR^d)$, for each $z\in\mathbf{P}$.

Clearly, $(\tilde{a}_z^w)^{-1}$ is the restriction of $(z+\overline{\tilde{A}})^{-1}$ to $\SSS^*(\RR^d)$. Now, observe that
\beqs
(\tilde{a}_z^w)^{-1}=(\mathrm{Id}-q_z^w\tilde{a}_z^w) (\tilde{a}_z^w)^{-1}(\mathrm{Id}-\tilde{a}_z^wq_z^w)+
(\mathrm{Id}-q_z^w\tilde{a}_z^w)q_z^w+q_z^w,
\eeqs
as operators on $\SSS^*(\RR^d)$. Proposition \ref{uniformparam} together with Lemma \ref{lemconofp} $(ii)$ yields that the set
$\{(1+|z|)(\mathrm{Id}-q_z^w\tilde{a}_z^w)q_z^w|\,
z\in\mathbf{P}\}$ is equicontinuous $*$-regularising.
Proposition \ref{uniformparam} implies that for each $z\in\mathbf{P}$, the
operator
$(\mathrm{Id}-q_z^w\tilde{a}_z^w)(\tilde{a}_z^w)^{-1}(\mathrm{Id}-\tilde{a}_z^wq_z^w)$
extends to a $*$-regularising operator. Thus, for each
$z\in\mathbf{P}$, $(\tilde{a}_z^w)^{-1}$ extends to a continuous
operator on $\SSS'^*(\RR^d)$. Since
$\sigma(\overline{\tilde{A}})\subseteq \{\lambda\in\RR|\, \lambda>
1\}$ and $\overline{\tilde{A}}$ is self-adjoint, \cite[Theorem
1.3.5, p. 21]{Fractionalpowersbook} yields that
$\overline{\tilde{A}}$ is sectorial with spectral angle $0$, and
this in turn yields that for each $0<\delta\leq 1$ there exists
$C_{\delta}>0$ such that
\beq\label{inss1}
\|(z+\overline{\tilde{A}})^{-1}\|\leq C_{\delta}/|z|,
\eeq
for all $z\in\{\zeta\in\CC\backslash\{0\}|\,
-\pi+\delta\leq\arg\zeta\leq \pi-\delta\}$. Denote
the particular constant for which (\ref{inss1}) holds true on
$\mathbf{P}_*=\mathbf{P}\backslash\{0\}$ by $\tilde{C}$. Since
$\sigma(\overline{\tilde{A}})\subseteq \{\lambda\in\RR|\, \lambda>
1\}$, we have $\|(z+\overline{\tilde{A}})^{-1}\|\leq C'$, for all $|z|\leq
1$. Now, Proposition \ref{uniformparam} yields that
$\{|z|(\mathrm{Id}-q_z^w\tilde{a}_z^w)(\tilde{a}_z^w)^{-1}(\mathrm{Id}-\tilde{a}_z^wq_z^w)|\,
z\in\mathbf{P}\}$ and
$\{(\mathrm{Id}-q_z^w\tilde{a}_z^w)(\tilde{a}_z^w)^{-1}(\mathrm{Id}-\tilde{a}_z^wq_z^w)|\,
z\in\mathbf{P}\}$ are equicontinuous $*$-regularising and thus,
the same holds for
$\{(1+|z|)(\mathrm{Id}-q_z^w\tilde{a}_z^w)(\tilde{a}_z^w)^{-1} (\mathrm{Id}-\tilde{a}_z^wq_z^w)|\,
z\in\mathbf{P}\}$ as well. Denoting
$S_z=(\mathrm{Id}-q_z^w\tilde{a}_z^w)(\tilde{a}_z^w)^{-1} (\mathrm{Id}-\tilde{a}_z^wq_z^w)+(\mathrm{Id}-q_z^w\tilde{a}_z^w)q_z^w$, we have $(\tilde{a}_z^w)^{-1}=q_z^w+S_z$. These facts, together with Lemma \ref{lemconofp} $(ii)$, prove the following result.

\begin{lemma}\label{kkr17}
The operators $(\tilde{a}_z^w)^{-1}$, $z\in\mathbf{P}$, are continuous on $\SSS^*(\RR^d)$ and they extend to
continuous operators on $\SSS'^*(\RR^d)$. The set $\{(1+|z|)(\tilde{a}_z^w)^{-1}|\,z\in\mathbf{P}\}$ is equicontinuous in $\mathcal{L}(\SSS^*(\RR^d),\SSS^*(\RR^d))$ and in $\mathcal{L}(\SSS'^*(\RR^d),\SSS'^*(\RR^d))$. Furthermore, for each $z\in\mathbf{P}$, $(\tilde{a}_z^w)^{-1}$ is exactly the restriction of $(z+\overline{\tilde{A}})^{-1}$ to $\SSS^*(\RR^d)$.
\end{lemma}

Consider now the uniformly bounded $C_0$-semigroup
$\tilde{T}(t)=e^{-tc}T(t)$, $t\geq 0$. Clearly, its infinitesimal generator
is $-\overline{\tilde{A}}$. Hence, \cite[Theorem 5.2 (c), p.
61]{pazy} proves that $\{\tilde{T}(t)\}_{t\geq 0}$ is analytic (cf. (\ref{inss1})) and
\cite[Theorem 7.7, p. 30]{pazy} yields
\beq\label{sstt151}
\tilde{T}(t)=\frac{1}{2\pi i}\int_{\Lambda}
e^{zt}(z+\overline{\tilde{A}})^{-1}dz,\,\, t>0,
\eeq
where $\Lambda$ is a smooth curve in $\{\zeta\in\CC\backslash\{0\}|\,-\pi+\delta\leq\arg\zeta\leq \pi-\delta\}$ for any $0<\delta<1$, running from $\infty e^{-i\theta}$ to $\infty e^{i\theta}$ for arbitrary but fixed $\pi/2<\theta<\pi-\delta$ and the integral is absolutely convergent for $t>0$ in $\mathcal{L}_b(L^2(\RR^d),L^2(\RR^d))$ (cf. (\ref{inss1})).

\begin{proposition}\label{smsemig}
For each $t\geq 0$, $\tilde{T}(t)\in\mathcal{L}(\SSS^*(\RR^d),\SSS^*(\RR^d))$. Moreover, the mapping $t\mapsto \tilde{T}(t)$ belongs to $ C^{\infty}([0,\infty);\mathcal{L}_b(\SSS^*(\RR^d),\SSS^*(\RR^d)))$ and its derivatives are given by $(d^k/dt^k)\tilde{T}(t)=(-1)^k(\tilde{a}^w)^k\tilde{T}(t)$, $t\geq 0$, $k\in\ZZ_+$.
\end{proposition}

\begin{proof} Because of the analyticity of $z\mapsto (z+\overline{\tilde{A}})^{-1}$, we can shift the path of integration without changing the value of the integral in (\ref{sstt151}) to the curve $\tilde{\Lambda}=\Lambda_1\cup\Lambda_2\cup\Lambda_3$, where $\Lambda_1=\{re^{-i3\pi/4}|\, 1\leq r<\infty\}$, $\Lambda_2=\{e^{i\theta}|\, -3\pi/4\leq \theta\leq3\pi/4\}$ and
$\Lambda_3=\{re^{i3\pi/4}|\, 1\leq r<\infty\}$. Clearly $\tilde{\Lambda}\subseteq \mathbf{P}_*$. For $\varphi\in \SSS^*(\RR^d)$, we have
\beqs
\tilde{T}(t)\varphi&=&\frac{1-i}{2\pi}\int_{1/\sqrt{2}}^{\infty}e^{-rt-irt} (\tilde{a}_{-r(1+i)}^w)^{-1}\varphi dr +\frac{1}{2\pi}\int_{-3\pi/4}^{3\pi/4}e^{te^{i\theta}}e^{i\theta} (\tilde{a}_{e^{i\theta}}^w)^{-1}\varphi d\theta\\
&{}&+ \frac{1+i}{2\pi}\int_{1/\sqrt{2}}^{\infty}e^{-rt+irt} (\tilde{a}_{-r(1-i)}^w)^{-1}\varphi dr\\
&=&I_1(t,\varphi)+I_2(t,\varphi)+I_3(t,\varphi),
\eeqs
with absolutely convergent integrals for $t>0$ in $L^2(\RR^d)$ (cf. (\ref{inss1}); recall $(\tilde{a}^w_z)^{-1}\varphi=(z+\overline{\tilde{A}})^{-1}\varphi$, for $z\in\mathbf{P}$, $\varphi\in\SSS^*(\RR^d)$). By the properties of the Bochner integral, for each $g\in L^2(\RR^d)$, we have
\beq\label{eqffp}
\langle g,I_1(t,\varphi)\rangle&=&\frac{1-i}{2\pi}\int_{1/\sqrt{2}}^{\infty}e^{-rt-irt} \langle g,(\tilde{a}_{-r(1+i)}^w)^{-1}\varphi\rangle dr.
\eeq
Our immediate goal is to prove $I_1(t,\varphi)\in\SSS^*(\RR^d)$ for each $t>0$ and $\varphi\in\SSS^*(\RR^d)$. Thus, fix $t>0$ and denote
\beq\label{kk117ks}
C_t=\int_{1/\sqrt{2}}^{\infty}e^{-rt} dr>0.
\eeq
Let $\varphi\in\SSS^*(\RR^d)$ and $\varepsilon>0$ be arbitrary but fixed. By Lemma \ref{kkr17}, the set $\tilde{H}=\{(1+|z|)(\tilde{a}_z^w)^{-1}|\,z\in\mathbf{P}\}$ is equicontinuous in $\mathcal{L}(\SSS^*(\RR^d),\SSS^*(\RR^d))$ and hence $B=\{(1+r\sqrt{2})(\tilde{a}_{-r(1+i)}^w)^{-1}\varphi|\,r\geq1/\sqrt{2}\}$ is bounded in $\SSS^*(\RR^d)$. Thus, the absolute polar of $(C_t/\varepsilon)B$, which we denote by $W=((C_t/\varepsilon)B)^{\circ}$, is a neighbourhood of zero in $\SSS'^*(\RR^d)$. Hence, employing (\ref{eqffp}) for $g\in W\cap L^2(\RR^d)$, we have
\beqs
|\langle g,I_1(t,\varphi)\rangle| &\leq&\int_{1/\sqrt{2}}^{\infty}e^{-rt} \left|\langle g,(\tilde{a}_{-r(1+i)}^w)^{-1}\varphi\rangle\right| dr\leq \varepsilon.
\eeqs
Thus, the mapping $g\mapsto \langle g,I_1(t,\varphi)\rangle$, $L^2(\RR^d)\rightarrow\CC$, is continuous when we equip $L^2(\RR^d)$ with the topology induced on it by $\SSS'^*(\RR^d)$. Hence $g\mapsto \langle g,I_1(t,\varphi)\rangle$ can be continuously extended to a functional on $\SSS'^*(\RR^d)$, i.e. $I_1(t,\varphi)\in\SSS^*(\RR^d)$. Let $g\in\SSS'^*(\RR^d)$. There exist $g_j\in L^2(\RR^d)$, $j\in\ZZ_+$, such that $g_j\rightarrow g$ in $\SSS'^*(\RR^d)$ ($L^2(\RR^d)$ is sequentially dense in $\SSS'^*(\RR^d)$). The function $r\mapsto \langle g,(\tilde{a}_{-r(1+i)}^w)^{-1}\varphi\rangle$, $[1/\sqrt{2},\infty)\rightarrow \CC$, is measurable since it is the pointwise limit of the sequence of continuous functions $r\mapsto \langle g_j,(\tilde{a}_{-r(1+i)}^w)^{-1}\varphi\rangle$, $[1/\sqrt{2},\infty)\rightarrow\CC$. Because of the equicontinuity of $\tilde{H}$ and the fact that $\{g_j|\, j\in\ZZ_+\}$ is bounded in $\SSS'^*(\RR^d)$, we can conclude the existence of $C'>0$ such that $|\langle g_j,(\tilde{a}_{-r(1+i)}^w)^{-1}\varphi\rangle|\leq C'$, for all $r\in[1/\sqrt{2},\infty)$, $j\in\ZZ_+$. Applying the dominated convergence theorem to (\ref{eqffp}) with $g_j$ in place of $g$, we can conclude that (\ref{eqffp}) is valid for $g\in\SSS'^*(\RR^d)$. Next, we prove that for each $t>0$, the mapping $\varphi\mapsto I_1(t,\varphi)$, $\SSS^*(\RR^d)\rightarrow\SSS^*(\RR^d)$, is continuous. Let $V$ be a closed convex circled neighbourhood of zero in $\SSS^*(\RR^d)$, which, without loss of generality, we can assume to be the absolute polar $B'^{\circ}$ of a bounded set $B'$ in $\SSS'^*(\RR^d)$. Since
\beq\label{krt1155991111}
\tilde{H}'=\{(1+|z|)\,\, {}^t((\tilde{a}_z^w)^{-1})|\,z\in\mathbf{P}\}
\eeq
is equicontinuous in $\mathcal{L}(\SSS'^*(\RR^d),\SSS'^*(\RR^d))$ (cf. Lemma \ref{kkr17} and \cite[Theorem 6, p. 138]{kothe2}), the set $B'_1=\{(1+|z|)\,\, {}^t((\tilde{a}_z^w)^{-1})g|\,z\in\mathbf{P},\, g\in B'\}$ is bounded in $\SSS'^*(\RR^d)$. Hence $V_1=(C_tB'_1)^{\circ}$ is a neighbourhoods of zero in $\SSS^*(\RR^d)$ (see (\ref{kk117ks}) for the definition of $C_t$). Employing (\ref{eqffp}), one easily verifies $|\langle g,I_1(t,\varphi)\rangle|\leq 1$, for all $\varphi\in V_1$ and $g\in B'$, which proves the desired continuity. Next, we prove that the mapping $t\mapsto I_1(t,\cdot)$ belongs to $ C((0,\infty);\mathcal{L}_b(\SSS^*(\RR^d),\SSS^*(\RR^d)))$. Fix $t_0>0$ and let $\delta>0$ be small enough such that $[t_0-\delta,t_0+\delta]\subseteq (0,\infty)$. Consider the following subset of $\mathcal{L}(\SSS^*(\RR^d),\SSS^*(\RR^d))$:
\beqs
H_1=\{\varphi\mapsto I_1(t,\varphi)|\, t\in[t_0-\delta,t_0+\delta]\}.
\eeqs
Employing (\ref{eqffp}) together with the fact that $\tilde{H}'$ is equicontinuous in $\mathcal{L}(\SSS'^*(\RR^d),\SSS'^*(\RR^d))$ (see (\ref{krt1155991111}) for the definition of $\tilde{H}'$), one can easily derive that $H_1$ is a bounded set in $\mathcal{L}_{\sigma}(\SSS^*(\RR^d),\SSS^*(\RR^d))$ and hence equicontinuous ($\SSS^*(\RR^d)$ is barrelled). Fix $\varphi\in\SSS^*(\RR^d)$ and a neighbourhood of zero $V$ in $\SSS^*(\RR^d)$ for which we may assume that it is the absolute polar $B'^{\circ}$ of a convex circled bounded subset $B'$ of $\SSS'^*(\RR^d)$. Let $1\leq C<\infty$ be large enough such that $C\geq \sup\{|\langle g,(\tilde{a}^w_z)^{-1}\varphi\rangle||\, z\in\mathbf{P},\, g\in B'\}$. Then, employing (\ref{eqffp}), we have
\beqs
\sup_{g\in B'}|\langle g, I_1(t,\varphi)-I_1(t_0,\varphi)\rangle|\leq C\int_{1/\sqrt{2}}^{\infty} e^{-r(t_0-\delta)}\left|e^{-r(t-t_0+\delta)}e^{-irt}-e^{-r\delta}e^{-irt_0}\right| dr,
\eeqs
for all $t\in[t_0-\delta,t_0+\delta]$. The dominated convergence theorem implies that there exists $0<\varepsilon<\delta$ such that $I_1(t,\varphi)-I_1(t_0,\varphi)\in B'^{\circ}=V$, for all $t\in[t_0-\varepsilon,t_0+\varepsilon]$. Hence $I_1(t,\cdot)\rightarrow I_1(t_0,\cdot)$, as $t\rightarrow t_0$, in $\mathcal{L}_{\sigma}(\SSS^*(\RR^d),\SSS^*(\RR^d))$. As $H_1$ is equicontinuous, the Banach-Steinhaus theorem \cite[Theorem 4.5, p. 85]{Sch} implies that the convergence holds in the topology of precompact convergence and, as $\SSS^*(\RR^d)$ is Montel, it also holds in $\mathcal{L}_b(\SSS^*(\RR^d),\SSS^*(\RR^d))$. This proves the continuity of the mapping $t\mapsto I_1(t,\cdot)$, $(0,\infty)\rightarrow \mathcal{L}_b(\SSS^*(\RR^d),\SSS^*(\RR^d))$.

In an analogous fashion one proves that for each $t>0$, the mappings $\varphi\mapsto I_2(t,\varphi)$ and $\varphi\mapsto I_3(t,\varphi)$ belong to $\mathcal{L}(\SSS^*(\RR^d),\SSS^*(\RR^d))$ and the mappings $t\mapsto I_2(t,\cdot)$ and $t\mapsto I_3(t,\cdot)$, $(0,\infty)\rightarrow\mathcal{L}_b(\SSS^*(\RR^d),\SSS^*(\RR^d))$, are continuous.

Thus, we obtain $\tilde{T}(t)\in\mathcal{L}(\SSS^*(\RR^d),\SSS^*(\RR^d))$, for each $t>0$, and also $t\mapsto \tilde{T}(t)\in C((0,\infty);\mathcal{L}_b(\SSS^*(\RR^d),\SSS^*(\RR^d)))$. Next, we prove the continuity at $0$. For each $t>0$, we shift the path of integration in (\ref{sstt151}) to $\tilde{\Lambda}_t=\tilde{\Lambda}_{1,t}\cup\tilde{\Lambda}_{2,t}\cup\tilde{\Lambda}_{3,t}$, where $\tilde{\Lambda}_{1,t}=\{re^{-i3\pi/4}|\, 1/t\leq r<\infty\}$, $\tilde{\Lambda}_{2,t}=\{t^{-1}e^{i\theta}|\, -3\pi/4\leq \theta\leq3\pi/4\}$ and
$\tilde{\Lambda}_{3,t}=\{re^{i3\pi/4}|\, 1/t\leq r<\infty\}$. Clearly $\tilde{\Lambda}_t\subseteq \mathbf{P}_*$. For $\varphi\in\SSS^*(\RR^d)$, we have
\beqs
\tilde{T}(t)\varphi&=&\frac{1-i}{2\pi}\int_{(t\sqrt{2})^{-1}}^{\infty}e^{-rt-irt} (\tilde{a}_{-r(1+i)}^w)^{-1}\varphi dr +\frac{1}{2\pi t}\int_{-3\pi/4}^{3\pi/4}e^{e^{i\theta}}e^{i\theta} (\tilde{a}_{t^{-1}e^{i\theta}}^w)^{-1}\varphi d\theta\\
&{}&+ \frac{1+i}{2\pi}\int_{(t\sqrt{2})^{-1}}^{\infty}e^{-rt+irt} (\tilde{a}_{-r(1-i)}^w)^{-1}\varphi dr\\
&=&\tilde{I}_1(t,\varphi)+\tilde{I}_2(t,\varphi)+\tilde{I}_3(t,\varphi).
\eeqs
Analogously as above, one establishes that, for each $t>0$ and $\varphi\in\SSS^*(\RR^d)$,  one has $\tilde{I}_1(t,\varphi),$ $\tilde{I}_2(t,\varphi),$ $\tilde{I}_3(t,\varphi)\in\SSS^*(\RR^d)$. By similar techniques as in the proof of the validity of (\ref{eqffp}) for $g\in\SSS'^*(\RR^d)$, one can prove that for each $g\in\SSS'^*(\RR^d)$, $\varphi\in\SSS^*(\RR^d)$ and $t>0$ we have
\beqs
\langle g,\tilde{I}_1(t,\varphi)\rangle&=&\frac{1-i}{2\pi}\int_{(t\sqrt{2})^{-1}}^{\infty}e^{-rt-irt} \langle g,(\tilde{a}_{-r(1+i)}^w)^{-1}\varphi\rangle dr,\\
\langle g,\tilde{I}_2(t,\varphi)\rangle&=&\frac{1}{2\pi t}\int_{-3\pi/4}^{3\pi/4}e^{e^{i\theta}}e^{i\theta} \langle g,(\tilde{a}_{t^{-1}e^{i\theta}}^w)^{-1}\varphi\rangle d\theta,\\
\langle g,\tilde{I}_3(t,\varphi)\rangle&=&\frac{1+i}{2\pi}\int_{(t\sqrt{2})^{-1}}^{\infty}e^{-rt+irt} \langle g,(\tilde{a}_{-r(1-i)}^w)^{-1}\varphi\rangle dr.
\eeqs
Fix $\varphi\in\SSS^*(\RR^d)$ and a bounded subset $B'$ of $\SSS'^*(\RR^d)$. The equicontinuity of $\tilde{H}'$ (cf. (\ref{krt1155991111})) implies the existence of $C'>0$ such that $|\langle g,(\tilde{a}_{-r(1+i)}^w)^{-1}\varphi\rangle|\leq C'/(1+r\sqrt{2})$, for all $g\in B'$, $r\in[0,\infty)$, and hence, a change of variables yields
\beqs
|\langle g,\tilde{I}_1(t,\varphi)\rangle|\leq \frac{C'}{\pi}\int_{1/\sqrt{2}}^{\infty} \frac{e^{-s}}{t+s\sqrt{2}}ds\leq C'\int_{1/\sqrt{2}}^{\infty} e^{-s}ds\leq C',
\eeqs
for all $g\in B'$, $t>0$. Similarly, there exists $C''>0$ such that $|\langle g,\tilde{I}_3(t,\varphi)\rangle|\leq C''$, for all $g\in B'$, $t>0$. Again, the equicontinuity of $\tilde{H}'$ yields the existence of $C'''>0$ such that $|\langle g,(\tilde{a}_{t^{-1}e^{i\theta}}^w)^{-1}\varphi\rangle|\leq C'''/(1+t^{-1})$, for all $g\in B'$, $\theta\in[-3\pi/4,3\pi/4]$, $t>0$. Hence
\beqs
|\langle g,\tilde{I}_2(t,\varphi)\rangle|\leq\frac{C'''}{2\pi t}\int_{-3\pi/4}^{3\pi/4} \frac{e^{\cos\theta}}{1+t^{-1}}d\theta\leq 3eC'''/4,
\eeqs
for all $g\in B'$, $t>0$. We conclude that there exists $C>0$ such that $|\langle g, \tilde{T}(t)\varphi\rangle|\leq C$, for all $g\in B'$, $t>0$. This proves that $\{\tilde{T}(t)|\, t>0\}$ is  bounded and hence equicontinuous in $\mathcal{L}(\SSS^*(\RR^d),\SSS^*(\RR^d))$. Consequently, the same holds for $\{\tilde{T}(t)|\, t\geq0\}$ (since $\tilde{T}(0)=\mathrm{Id}$). This immediately yields the equicontinuity of $\{\tilde{a}^w\tilde{T}(t)|\, t\geq0\}$ in $\mathcal{L}(\SSS^*(\RR^d),\SSS^*(\RR^d))$. Since $\{\tilde{T}(t)\}_{t\geq 0}$ is a $C_0$-semigroup with infinitesimal generator $-\overline{\tilde{A}}$, we have
\beqs
\tilde{T}(t)\varphi-\varphi=-\int_0^t \overline{\tilde{A}}\tilde{T}(s)\varphi ds=-\int_0^t \tilde{a}^w\tilde{T}(s)\varphi ds,\,\, \forall \varphi\in\SSS^*(\RR^d),\, t>0.
\eeqs
Employing the equicontinuity of $\{\tilde{a}^w\tilde{T}(t)|\, t\geq0\}$ in $\mathcal{L}(\SSS^*(\RR^d),\SSS^*(\RR^d))$ and using similar arguments as in the proof of the validity of (\ref{eqffp}) for $g\in\SSS'^*(\RR^d)$, one can prove that for each $g\in\SSS'^*(\RR^d)$, $\varphi\in\SSS^*(\RR^d)$ and $t>0$ we have
\beq\label{kkff1}
\langle g,\tilde{T}(t)\varphi-\varphi\rangle=-\int_0^t\langle g,\tilde{a}^w\tilde{T}(s)\varphi\rangle ds.
\eeq
For fixed $\varphi\in\SSS^*(\RR^d)$ and a bounded subset $B'$ of $\SSS'^*(\RR^d)$, the equicontinuity of the set $\{\tilde{a}^w\tilde{T}(t)|\, t\geq0\}$ in $\mathcal{L}(\SSS^*(\RR^d),\SSS^*(\RR^d))$ proves the existence of $C>0$ such that $|\langle g,\tilde{a}^w\tilde{T}(t)\varphi\rangle|\leq C$, for all $g\in B'$, $t\geq 0$. Thus, (\ref{kkff1}) yields $|\langle g,\tilde{T}(t)\varphi-\varphi\rangle|\leq Ct$, for all $g\in B'$, $t>0$. Hence, there exists $\varepsilon>0$ such that for all $0<t<\varepsilon$, $\tilde{T}(t)\varphi-\varphi\in B'^{\circ}$, which proves that $\tilde{T}(t)\rightarrow \tilde{T}(0)=\mathrm{Id}$, as $t\rightarrow 0^+$, in $\mathcal{L}_{\sigma}(\SSS^*(\RR^d),\SSS^*(\RR^d))$. Since $\{\tilde{T}(t)|\, t\geq0\}$ is equicontinuous in $\mathcal{L}(\SSS^*(\RR^d),\SSS^*(\RR^d))$, the Banach-Steinhaus theorem \cite[Theorem 4.5, p. 85]{Sch} yields that the convergence holds in the topology of precompact convergence and, as $\SSS^*(\RR^d)$ is Montel, the convergence also holds in $\mathcal{L}_b(\SSS^*(\RR^d),\SSS^*(\RR^d))$. This proves that $t\mapsto \tilde{T}(t)$ belongs to $ C([0,\infty);\mathcal{L}_b(\SSS^*(\RR^d),\SSS^*(\RR^d)))$.\\
\indent Observe now that for $t>t_0\geq 0$, $g\in\SSS'^*(\RR^d)$ and $\varphi\in\SSS^*(\RR^d)$, (\ref{kkff1}) implies
\begin{align}
\label{kkr1515}
&\frac{\langle g,\tilde{T}(t)\varphi-\tilde{T}(t_0)\varphi\rangle}{t-t_0}+\langle g,\tilde{a}^w\tilde{T}(t_0)\varphi\rangle
\\
&
\nonumber
\quad \quad =-\frac{1}{t-t_0}\int_{t_0}^t\langle g,\tilde{a}^w(\tilde{T}(s)-\tilde{T}(t_0))\varphi\rangle ds.
\end{align}
Let $B$ be a bounded subset of $\SSS^*(\RR^d)$ and $V$ a neighbourhood of zero in $\SSS^*(\RR^d)$. Consider the neighbourhood of zero $M(B,V)=\{S\in\mathcal{L}(\SSS^*(\RR^d),\SSS^*(\RR^d))|\, S(B)\subseteq V\}$ in $\mathcal{L}_b(\SSS^*(\RR^d),\SSS^*(\RR^d))$.  We may of course assume $V$ is the absolute polar $B'^{\circ}$ of a bounded set $B'$ in $\SSS'^*(\RR^d)$. Then $B'_1={}^t\,(\tilde{a}^w)B'$ is bounded in $\SSS'^*(\RR^d)$ and hence its absolute polar $V_1=B'^{\circ}_1$ is a neighbourhood of zero in $\SSS^*(\RR^d)$. Since $t\mapsto \tilde{T}(t)\in C([0,\infty);\mathcal{L}_b(\SSS^*(\RR^d),\SSS^*(\RR^d)))$, there exists $\varepsilon>0$ such that for all $t\in[t_0,t_0+\varepsilon]$, we have $\tilde{T}(t)-\tilde{T}(t_0)\in M(B,V_1)$. Thus, (\ref{kkr1515}) implies $(t-t_0)^{-1}(\tilde{T}(t)-\tilde{T}(t_0))+\tilde{a}^w\tilde{T}(t_0)\in M(B,V)$, for all $t\in(t_0,t_0+\varepsilon]$, i.e. the right derivative of $t\mapsto\tilde{T}(t)$, $[0,\infty)\rightarrow \mathcal{L}_b(\SSS^*(\RR^d),\SSS^*(\RR^d))$, is $-\tilde{a}^w\tilde{T}(t_0)$. Similarly, the left derivative at $t_0>0$ is $-\tilde{a}^w\tilde{T}(t_0)$. Hence, $t\mapsto\tilde{T}(t)$, $[0,\infty)\rightarrow \mathcal{L}_b(\SSS^*(\RR^d),\SSS^*(\RR^d))$, is differentiable and $(d/dt)\tilde{T}(t)=-\tilde{a}^w\tilde{T}(t)$. As $t\mapsto -\tilde{a}^w\tilde{T}(t)$ is continuous, $t\mapsto \tilde{T}(t)$ is of class $ C^1$ and now, the equality $(d/dt)\tilde{T}(t)=-\tilde{a}^w\tilde{T}(t)$ readily implies that $t\mapsto \tilde{T}(t)$ is in $ C^{\infty}([0,\infty);\mathcal{L}_b(\SSS^*(\RR^d),\SSS^*(\RR^d)))$ and $(d^k/dt^k)\tilde{T}(t)=(-1)^k(\tilde{a}^w)^k\tilde{T}(t)$, $k\in\ZZ_+$.
\end{proof}

As a direct consequence of the previous proposition we then have,

\begin{theorem}\label{semi-group-pse}
We have $T(t)\in\mathcal{L}(\SSS^*(\RR^d),\SSS^*(\RR^d))$ for each $t\geq 0$. Moreover, the mapping $t\mapsto T(t)$ belongs to $ C^{\infty}([0,\infty);\mathcal{L}_b(\SSS^*(\RR^d),\SSS^*(\RR^d)))$ and one has $(d^k/dt^k) T(t)=(-1)^k(a^w)^kT(t)$, $t\geq 0$, $k\in\ZZ_+$.
\end{theorem}

Since $T(t)$ solves (\ref{system3}) with $\tilde{\mathbf{K}}(t)=0$, we obtain
\beqs
(\mathbf{u}(t))^w\varphi-T(t)\varphi=\int_0^t T(t-s)\tilde{\mathbf{K}}(s)\varphi ds,\,\, \varphi\in\SSS^*(\RR^d).
\eeqs
Theorem \ref{semi-group-pse} then implies that for each $t>0$, the mapping $s\mapsto T(t-s)\tilde{\mathbf{K}}(s)$ belongs to $C^{\infty}([0,t];\mathcal{L}_b(\SSS'^*(\RR^d),\SSS^*(\RR^d)))$. For $t\geq0$ and $f\in\SSS'^*(\RR^d)$, define
\beqs
\mathbf{Q}(t)f=\int_0^t T(t-s)\tilde{\mathbf{K}}(s)fds\in L^2(\RR^d).
\eeqs
Similarly as in the proof of Proposition \ref{smsemig}, one verifies $\mathbf{Q}(t)f\in\SSS^*(\RR^d)$ and, for each $g\in\SSS'^*(\RR^d)$,
\beq\label{kfr97}
\langle g,\mathbf{Q}(t)f\rangle=\int_0^t \langle g,T(t-s)\tilde{\mathbf{K}}(s)f\rangle ds.
\eeq
Again, employing analogous techniques as in the proof of Proposition \ref{smsemig}, one can prove $f\mapsto\mathbf{Q}(t)f\in\mathcal{L}(\SSS'^*(\RR^d),\SSS^*(\RR^d))$, for each $t\geq 0$. Using the properties of $T(t)$ and $\tilde{\mathbf{K}}(t)$, one readily checks that the mapping $(t,s)\mapsto T(t-s)\tilde{\mathbf{K}}(s)$, $\{(t,s)\in\RR^2|\, 0\leq s\leq t\}\rightarrow \mathcal{L}_b(\SSS'^*(\RR^d),\SSS^*(\RR^d))$, is continuous. Hence, for each $C>0$, $\{T(t-s)\tilde{\mathbf{K}}(s)|\, 0\leq s\leq t\leq C\}$ is an equicontinuous subset of $\mathcal{L}(\SSS'^*(\RR^d),\SSS^*(\RR^d))$. Employing this fact together with (\ref{kfr97}) and the semigroup property of $T(t)$, one can prove that $t\mapsto\mathbf{Q}(t)$, $[0,\infty)\rightarrow \mathcal{L}_b(\SSS'^*(\RR^d),\SSS^*(\RR^d))$, is continuous. Now, reproducing the proof of \cite[Lemma 7.15]{PP1} verbatim, one gets the following result.

\begin{lemma}\label{smooth-ker-pe}
The mapping $t\mapsto\mathbf{Q}(t)$ belongs to $ C^{\infty}([0,\infty);\mathcal{L}_b(\SSS'^*(\RR^d),\SSS^*(\RR^d)))$.
\end{lemma}

Denoting the Weyl symbol of $\mathbf{Q}(t)$ by $Q(t,w)$, this lemma together with the property of symbols of operators in $\mathcal{L}(\SSS'^*(\RR^d),\SSS^*(\RR^d)))$ (cf. \cite[Propositions 2 and 3]{BojanP}) imply:

\begin{corollary}\label{ttk771133}
The mapping $t\mapsto Q(t,\cdot)$ belongs to $ C^{\infty}([0,\infty);\SSS^*(\RR^{2d}))$.
\end{corollary}

Notice that (\ref{est_heat_par}) together with $a(w)/\ln|w|\rightarrow +\infty$, as $|w|\rightarrow\infty$, ensures that $(\mathbf{u}(t))^w$ is trace-class for each $t>0$ (cf. \cite[Theorem 4.4.21, p. 190]{NR}). Now, Lemma \ref{smooth-ker-pe} ensures that $T(t)$ is also trace-class for $t>0$. As $T(t)$ are self-adjoint, we conclude $\mathrm{Tr}\, T(t)=\sum_{j=0}^{\infty}e^{-t\lambda_j}$. Thus,
\beqs
\sum_{j=0}^{\infty} e^{-t\lambda_j}=\frac{1}{(2\pi)^d}\int_{\RR^{2d}}u(t,w)dw-\frac{1}{(2\pi)^d}\int_{\RR^{2d}}Q(t,w)dw,\,\,\, t>0.
\eeqs
The second integral is $O(1)$ as $t\rightarrow 0^+$ (because of Corollary \ref{ttk771133}). Fix $n>d/\rho$, $n\in\ZZ_+$. Since $u_0(t,w)=e^{-tb(w)}$ and $b(w)=a(w)$ for $w$ outside of a compact neighbourhood of the origin, we have
\beqs
\sum_{j=0}^{\infty} e^{-t\lambda_j}&=&\frac{1}{(2\pi)^d}\int_{\RR^{2d}}e^{-ta(w)}dw+\frac{1}{(2\pi)^d}\sum_{j=1}^{n-1}\int_{\RR^{2d}}u_j(t,w)dw\\
&{}&+ \frac{1}{(2\pi)^d}\int_{\RR^{2d}}\left(u(t,w)-\sum_{j=0}^{n-1}u_j(t,w)\right)dw+O(1),\,\, t\rightarrow 0^+.
\eeqs
In view of the second estimate in Lemma \ref{rks75} (specialised for $n=0$ and $\alpha=0$), the very last integral is $O(1)$ as $t\rightarrow 0^+$. Lemma \ref{est_heat_ker} implies that there exists $C'>0$ such that $|u_j(t,w)|\leq Ce^{-\frac{t}{4}b(w)}\langle w\rangle^{-2\rho}$, for all $w\in\RR^{2d}$, $t\geq0$, $j=1,\ldots, n-1$. Using again $b=a$ except in a compact neighbourhood of $0$, we have
\beqs
\sum_{j=0}^{\infty} e^{-t\lambda_j}=\frac{1}{(2\pi)^d}\int_{\RR^{2d}}e^{-ta(w)}dw+ O\left(\int_{\RR^{2d}}\frac{e^{-\frac{t}{4}a(w)}}{\langle w\rangle^{2\rho}}dw\right)+O(1),\,\,\, t\to0^{+}.
\eeqs
We claim
\beq\label{heatparametrixasymptoticbound11}
\lim_{t\rightarrow 0^+}  \int_{\RR^{2d}}\frac{e^{-ta(w)/4}}{\langle w\rangle^{2\rho}}dw=\infty.
\eeq
To verify it, first notice that $a\in\Gamma^{*,\infty}_{A_p,\rho}(\RR^{2d})$ implies that there are $m,C>0$ (resp. for every $m>0$ there exists $C>0$) such that $a(w)\leq C e^{M(m|w|)}$, $\forall w\in\RR^{2d}$. Using this estimate (in the Roumieu case we can take $m=1$ with the corresponding $C>0$) and polar coordinates, we have
\beqs
\int_{\RR^{2d}}\frac{e^{-ta(w)/4}}{\langle w\rangle^{2\rho}}dw&\geq& \int_{\mathbb{S}^{2d-1}}\int_0^{\infty}\exp\left(-\frac{tCe^{M(m r)}}{4}\right)\frac{r^{2d-1}}{(1+r^2)^{\rho}}dr d\vartheta\\
&\geq& \frac{2\pi^d}{(d-1)!}\int_1^{\infty}\exp\left(-\frac{tCe^{M(m r)}}{4}\right)\frac{r}{1+r^2}dr.
\eeqs
Monotone convergence implies that the very last integral tends to $\infty$ as $t\rightarrow0^+$. We have shown:

\begin{theorem}\label{assymptoticsheatpara}
Let $a$ be a hypoelliptic real-valued symbol in $\Gamma^{*,\infty}_{A_p,\rho}(\RR^{2d})$ such that
$$
\lim_{|w|\to\infty} \frac{a(w)}{\ln |w|}=\infty.
$$
Then
\beq
\label{heatparametrixasymptoticbound}
\sum_{j=0}^{\infty} e^{-t\lambda_j}=\frac{1}{(2\pi)^d}\int_{\RR^{2d}}e^{-ta(w)}dw+ O\left(\int_{\RR^{2d}}\frac{e^{-\frac{t}{4}a(w)}}{\langle w\rangle^{2\rho}}dw\right),\,\,\, t\rightarrow 0^+.
\eeq
\end{theorem}

The next remark shows that (\ref{heatparametrixasymptoticbound}) remains valid for hypoelliptic symbols of finite order.

\begin{remark}\label{forthefiniteord}
Let $a\in \Gamma^m_{\rho}(\RR^{2d})$ be a hypoelliptic real-valued symbol such that $a(w)\geq c\langle w\rangle^{\delta}$ for some $\delta>0$, $\forall|w|\geq c$, and consider its heat parametrix $(\mathbf{u}(t))^w=(u(t,\cdot))^w$ as constructed in Remark \ref{heat-para-fin} and the $C_0$-semigroup $\{T(t)\}_{t\geq 0}$ as given by (\ref{semigroup11}). The fact $t\mapsto T(t)\in  C^{\infty}([0,\infty);\mathcal{L}_b(\SSS(\RR^d),\SSS(\RR^d)))$ can be proved far more easily in the distributional setting. To verify this, first notice that $(a^w)^j$ is hypoelliptic for each $j\in\ZZ_+$ and denote its symbol by $a_j\in\Gamma^{jm}_{\rho}(\RR^{2d})$. Clearly $|a_j(w)|\geq \langle w\rangle^{\delta j}$ away the origin. For each $\varphi\in\SSS(\RR^d)$, $t\geq 0$ and $j\in\ZZ_+$, we have $(a^w)^jT(t)\varphi=T(t)(a^w)^j\varphi\in L^2(\RR^d)$. Because of \cite[Theorem 2.1.16, p. 76]{NR}, $T(t)\varphi$ belongs to all Sobolev spaces $H^k_{\Gamma}(\RR^d)$, $k\in\ZZ_+$, and thus $T(t)\varphi\in \SSS(\RR^d)$. Now, the closed graph theorem yields $T(t)\in\mathcal{L}(\SSS(\RR^d),\SSS(\RR^d))$, $t\geq 0$. Since $\ds \SSS(\RR^d)=\lim_{\substack{\longleftarrow\\ k\rightarrow\infty}}H^k_{\Gamma}(\RR^d)$, in order to prove that $t\mapsto T(t)$ is right continuous at $t_0$ it is enough to prove that for each $k\in\ZZ_+$, $\varepsilon>0$ and bounded subset $B$ of $\SSS(\RR^d)$, there exists $\eta>0$ such that $\|T(t)\varphi-T(t_0)\varphi\|_{H^k_{\Gamma}}\leq \varepsilon$, $\forall t\in(t_0,t_0+\eta)$, $\forall\varphi\in B$. The a priori estimate in \cite[Theorem 2.1.16, p. 76]{NR} yields that there exist $C>0$ and $j\in\ZZ_+$ such that\\
\\
$\|T(t)\varphi-T(t_0)\varphi\|_{H^k_{\Gamma}}$
\beqs
\leq C\|T(t_0)\|_{\mathcal{L}_b(L^2(\RR^d))} (\|(T(t-t_0)-\mathrm{Id})(a^w)^j\varphi\|_{L^2}+ \|(T(t-t_0)-\mathrm{Id})\varphi\|_{L^2}).
\eeqs
Since $T(t)\rightarrow \mathrm{Id}$ in $\mathcal{L}_p(L^2(\RR^d),L^2(\RR^d))$ (by the Banach-Steinhaus theorem; $\{T(t)\}_{t\geq 0}$ is a $C_0$-semigroup) and $B$ and $(a^w)^j(B)$ are precompact in $\SSS(\RR^d)$ and hence also in $L^2(\RR^d)$, we obtain that $t\mapsto T(t)$ is right continuous at $t_0$. Similarly, one proves that it is left continuous. The same a priori estimate proves that the set $H=\{(t-t_0)^{-1}(T(t)-T(t_0))|\, t\in([t_0-1,t_0+1]\backslash\{t_0\})\cap [0,\infty)\}$ is bounded in $\mathcal{L}_{\sigma}(\SSS(\RR^d),\SSS(\RR^d))$, hence equicontinuous. Again, the same a priori estimate proves $(t-t_0)^{-1}(T(t)-T(t_0))\rightarrow -a^wT(t_0)$ in $\mathcal{L}_{\sigma}(\SSS(\RR^d),\SSS(\RR^d))$ and, as $H$ is equicontinuous, the Banach-Steinhaus theorem \cite[Theorem 4.5, p. 85]{Sch} gives the limit in the topology of precompact convergence. As $\SSS(\RR^d)$ is Montel, the limit holds in the strong topology. This immediately yields $t\mapsto T(t)\in C^{\infty}([0,\infty);\mathcal{L}_b(\SSS(\RR^d),\SSS(\RR^d)))$. Now one can obtain in the same way as above the validity of Lemma \ref{smooth-ker-pe} and Corollary \ref{ttk771133} in this case as well (of course, with $\SSS(\RR^d)$ and $\SSS'(\RR^d)$ in place of $\SSS^*(\RR^d)$ and $\SSS'^*(\RR^d)$).\\
\indent Using the estimates for $u(t,w)$ and $u_j(t,w)$ given in Remark \ref{heat-para-fin}, one readily obtains (\ref{heatparametrixasymptoticbound11}) and the asymptotic estimate (\ref{heatparametrixasymptoticbound}) from Theorem \ref{assymptoticsheatpara} in the finite order case too.
\end{remark}

\section{The Weyl asymptotic formula for infinite order $\Psi$DOs. Part II: proofs of the main results}\label{proofWeylasymp}

We now present the proofs of Theorems \ref{Weylth1}, \ref{Weylth2}, \ref{Weylth3}, and Corollary \ref{Weylformulac1}. In the sequel, we also use Vinogradov's notation  for $O$-estimates, namely, $g_{1}(t)\ll g_{2}(t)$ as an alternative way of writing $g_{1}(t)=O(g_{2}(t))$.\\
\indent We first make some comments that apply to all cases simultaneously. A preliminary observation is that $f(y)/y^{\delta} \to\infty$ as $y\to\infty$ for any $0<\delta<\ds\liminf_{y\to\infty}yf'(y)/f(y)$ as follows by integrating (\ref{weyleq8}) which holds in the three cases. It then follows from (\ref{weyleq3}), (\ref{weyleq6}), or (\ref{weyleq9}) that  $a(w)/| w |^{\delta} \to\infty$ as $w\to\infty$.
Incidentally, this also implies that $f'(y)>0$ a.e. on $[Y_1,\infty)$, for some large enough $Y_1\geq Y$ and additionally $f(y)>1$ on $[Y_1,\infty)$. Without loss of generality, we may assume $Y_1=Y>1$. We conclude that $\sigma$ is absolutely continuous on every compact interval contained in $[f(Y),\infty)$. We extend $\sigma$ to $[0,f(Y)]$ as a positive non-decreasing absolutely continuous function with $\sigma(\lambda)=1$ near $\lambda=0$. Note also that $\sigma(\lambda)\to\infty$ as $\lambda\to\infty$. We now derive some  regular variation properties of $\sigma$.

For Theorems \ref{Weylth1} and \ref{Weylth2}, and Corollary \ref{Weylformulac1}, we combine (\ref{weyleq2}) and (\ref{weyleq5}) into
\beq
\label{weyleq14}
\lim_{y\to\infty} \frac{yf'(y)}{f(y)}=\beta\in (0,\infty].
\eeq
Let us verify that (\ref{weyleq14}) implies that $\sigma$ is a Karamata regular varying function \cite{BGT} with index of regular variation $2d/\beta$ ($=0$ if $\beta=\infty$), that is, that
\begin{equation}\label{weyleq15}
\lim_{\lambda\to\infty}\frac{\sigma(\alpha\lambda)}{\sigma(\lambda)}=\alpha^{\frac{2d}{\beta}},
\end{equation}
uniformly for $\alpha$ in compact subsets of $(0,\infty)$. In fact, we have that
$$\eta(\lambda)=\frac{\lambda \sigma'(\lambda)}{\sigma(\lambda)}=2d\frac{f(f^{-1}(\lambda))}{f^{-1}(\lambda)f'(f^{-1}(\lambda))}\to \frac{2d}{\beta}, \quad \lambda\to\infty,$$
and
$$\sigma(\lambda)=\exp{\left(\int_{0}^{\lambda}\frac{\eta(t)}{t}dt\right)}$$
 for all $\lambda$ (note that $\eta(t)$ vanishes for $t$ near 0). This easily yields (\ref{weyleq15}).

Similarly, the hypothesis (\ref{weyleq8}) and the fact that $\sigma$ is increasing imply that there are $\nu,C_{1}>0$ such that
\begin{equation}\label{weyleq15.1}
\sigma(\alpha\lambda)/\sigma(\lambda)\leq C_1 (\alpha+1)^{\nu}, \quad \forall {\alpha,\lambda>0}.
\end{equation}
In fact, we may take any $\nu>0$ such that $2d/\nu<\beta'=\liminf_{y\to\infty}yf'(y)/f(y)$. For $\nu$ in this range, the inequality can be refined for large $\lambda$. Indeed, there is $\lambda_{0}=\lambda_{0}(\nu)$ such that
\begin{equation}\label{weyleq15.2}
\sigma(\alpha\lambda)/\sigma(\lambda)\leq \alpha^{\nu}, \quad \forall \alpha\geq1,\: \lambda\geq \lambda_0.
\end{equation}

The next starting point is the formula \eqref{heatparametrixasymptoticbound} from Theorem \ref{assymptoticsheatpara}, which holds under all our three sets of hypotheses (see Remark \ref{forthefiniteord} for the finite order case). As there are only finitely many possibly negative eigenvalues, we obtain (cf. (\ref{heatparametrixasymptoticbound11}))
\begin{equation}
\label{Weyleq16}
\int_{0}^{\infty}e^{-t\lambda }dN(\lambda)=\frac{1}{(2\pi)^d}\int_{\RR^{2d}}e^{-ta(w)}dw+ O\left(\int_{\RR^{2d}}\frac{e^{-ta(w)/4}}{\langle w\rangle^{2\rho}}dw\right), \quad  t\to0^{+}.
\end{equation}

\begin{proof}[Proof of Theorem \ref{Weylth1}] Let $\varepsilon>0$ be arbitrary but fixed and set
\beq
\label{eqproofc1Weylextra}
C'= \frac{1}{2d} \int_{\mathbb{S}^{2d-1}}\frac{d\vartheta}{(\Phi(\vartheta))^{2d}}\: .
\eeq
Using polar coordinates and the lower bound from \eqref{weyleq3}, we have that
 \begin{align*}\int_{\RR^{2d}}e^{-ta(w)}dw&
 \leq \int_{\mathbb{S}^{2d-1}}\int_{B_{\varepsilon}}^{\infty}r^{2d-1}e^{-c_{\varepsilon}t f((1-\varepsilon)r \Phi(\vartheta))}drd\vartheta
 + \int_{|w|\leq B_{\varepsilon}}e^{-ta(w)}dw
 \\
 &
 =\int_{\mathbb{S}^{2d-1}}\int_{B_{\varepsilon}}^{\infty}r^{2d-1}e^{-c_{\varepsilon}t f((1-\varepsilon)r \Phi(\vartheta))}drd\vartheta+ O_{\varepsilon}(1)
  \\
 &
 = (1-\varepsilon)^{-2d} C'\int_{f(Y)}^{\infty}e^{-c_\varepsilon t\lambda}\sigma'(\lambda)d\lambda+ O_{\varepsilon}(1)
 \\
 &
 =
(1-\varepsilon)^{-2d} C' \int_{0}^{\infty}e^{-\lambda }\sigma(\lambda/(c_\varepsilon t))d\lambda+ O_{\varepsilon}(1),
 \end{align*}
where we have used the change of variables $\lambda=f((1-\varepsilon)r \Phi(\vartheta))$ which gives
$$ r^{2d-1}\: dr= \frac{1}{2d} \left(\frac{1}{(1-\varepsilon)\Phi(\vartheta)}\right)^{2d}\sigma'(\lambda)d\lambda.$$
 Since $\sigma$ is slowly varying (i.e. $\sigma(\alpha\lambda)/\sigma(\lambda)\to1$ as $\lambda\to\infty$),
$$\int_{0}^{\infty}e^{-\lambda }\sigma(\lambda/(c_{\varepsilon}t))d\lambda\sim \sigma\left(1/t\right), \quad t\to0^{+},$$
as follows from the Lebesgue dominated convergence theorem (the bound \eqref{weyleq15.1} holds here for every $\nu>0$ and $C_1$ depending only on $\nu$).
Thus,
$$
\limsup_{t\to0^{+}} \frac{1}{\sigma(1/t)}\int_{\RR^{2d}}e^{-ta(w)}dw\leq (1-\varepsilon)^{-2d} C',
$$
because $\sigma(1/t)\to\infty$. But we can now take $\varepsilon\to 0^{+}$ to conclude
$$
\limsup_{t\to0^{+}} \frac{1}{\sigma(1/t)}\int_{\RR^{2d}}e^{-ta(w)}dw\leq C'.
$$
Similarly,
$$\liminf_{t\to0^{+}} \frac{1}{\sigma(1/t)}\int_{\RR^{2d}}e^{-ta(w)}dw\geq C';$$
therefore,
$$
\frac{1}{(2\pi)^d}\int_{\RR^{2d}}e^{-ta(w)}dw\sim \frac{C'}{(2\pi)^{d}} \sigma(1/t), \quad t\to0^{+}.
$$
On the other hand, a small computation along the same lines as the above one shows that
\begin{equation}
\label{weyleq17}
\int_{\RR^{2d}}\frac{e^{-ta(w)/4}}{\langle w\rangle^{2\rho}}dw\ll \sigma(1/t)^{1-\frac{\rho}{d}}=o(\sigma(1/t)).
\end{equation}
Inserting all this into (\ref{Weyleq16}), we conclude
$$\int_{0}^{\infty}e^{-t\lambda}dN(\lambda)\sim  \frac{C'}{(2\pi)^{d}}  \sigma\left(\frac{1}{t}\right), \quad t\to0^{+},$$
and (\ref{weyleq4}) follows from the well known Karamata Tauberian theorem \cite[ Theorem 1.7.1, p. 37]{BGT} (see also \cite[ Theorem 8.1, p. 193]{Korevaarbook}).\\
\indent Using (\ref{weyleq4}) and employing a classical argument (see e.g. \cite[Proposition 4.6.4, p. 198]{NR}, the same proof works fine in our case), we obtain that
\beq\label{ktfhhh113}
\sigma(\lambda_j)\sim \frac{(2\pi)^{d}}{C'} j, \quad j\to\infty.
\eeq
Notice that (\ref{weyleq4.1}) is equivalent to (\ref{ktfhhh113}). Finally, \eqref{weyleq4.2} follows from \eqref{weyleq4.1} and
$$
\frac{f(\alpha y)}{f(\alpha' y)}=\exp\left(\int_{\alpha' y}^{\alpha y}\frac{f'(t)}{f(t)}dt\right)\to \infty, \quad y\to\infty,
$$
valid for every $\alpha'<\alpha$ because of (\ref{weyleq2}). This completes the proof of Theorem \ref{Weylth1}.
\end{proof}

\begin{proof}[Proof of Theorem \ref{Weylth2}]
\indent Pick $\varepsilon>0$ and find $B$ so large that
$$(1-\varepsilon)f(r)\Phi(\vartheta)\leq a(r\vartheta)\leq (1+\varepsilon)f(r)\Phi(\vartheta)$$
for all $\vartheta\in \mathbb{S}^{2d-1}$ and $r>B$. Note that $\Phi$ is continuous and thus $\Phi(\vartheta)$ stays on a compact subset of $(0,\infty)$. Using that \eqref{weyleq15} is valid uniformly for $\alpha$ on compact subsets of $(0,\infty)$, we then obtain,
 \begin{align*}\frac{1}{\sigma(1/t)}\int_{\RR^{2d}}e^{-ta(w)}dw&\leq \frac{1}{\sigma(1/t)}\int_{\mathbb{S}^{2d-1}}\int_{0}^{\infty}e^{-(1-\varepsilon)t\Phi(\vartheta) f(r)}r^{2d-1}drd\vartheta+ o_{\varepsilon}(1)
 \\&
 = \frac{1}{2d}\int_{0}^{\infty}e^{-\lambda }\left(\int_{\mathbb{S}^{2d-1}}
 \frac{\sigma(\lambda/((1-\varepsilon)\Phi(\vartheta)t))}{\sigma(1/t)}d\vartheta
 \right)d\lambda+ o_{\varepsilon}(1)
  \\
 &=
 \frac{\int_{0}^{\infty} e^{-\lambda} \lambda^{2d/\beta}d\lambda}{2d(1-\varepsilon)^{2d/\beta}}\left(\int_{\mathbb{S}^{2d-1}}\frac{d\vartheta}{(\Phi(\vartheta))^{2d/\beta}}\right)+o_{\varepsilon}(1)
 \\
 &=
 \frac{\Gamma\left(1+\frac{2d}{\beta}\right)}{2d(1-\varepsilon)^{2d/\beta}}\int_{\mathbb{S}^{2d-1}}\frac{d\vartheta}{(\Phi(\vartheta))^{2d/\beta}}+o_{\varepsilon}(1), \quad t\to0^{+}.
 \end{align*}
Taking first $t\to0^{+}$ and then $\varepsilon\to 0^{+}$, we conclude that
$$
\limsup_{t\to0^{+}} \frac{1}{\sigma(1/t)}\int_{\RR^{2d}}e^{-ta(w)}dw\leq  \frac{\Gamma\left(1+\frac{2d}{\beta}\right)}{2d}\int_{\mathbb{S}^{2d-1}}\frac{d\vartheta}{(\Phi(\vartheta))^{2d/\beta}}.
$$
The estimate (\ref{weyleq17}) remains valid in this case too. A similar analysis for the limit inferior, together with (\ref{Weyleq16}) and (\ref{weyleq17}), leads to
$$
\int_{0}^{\infty}e^{-t\lambda }dN(\lambda)\sim \sigma(1/t)\: \frac{\pi \Gamma\left(1+\frac{2d}{\beta}\right)}{(2\pi)^{d+1}d} \int_{\mathbb{S}^{2d-1}}\frac{d\vartheta}{(\Phi(\vartheta))^{2d/\beta}}, \quad t\to0^{+}.
$$
We can apply once again the Karamata Tauberian theorem \cite{BGT,Korevaarbook} to conclude that (\ref{weyleq7}) holds.\\
\indent The classical argument quoted above in the proof of Theorem \ref{Weylth1} easily gives
$
\sigma(\lambda_{j})\sim j/C,$ $ j\to\infty,$
with $C= d^{-1}(2\pi)^{-d-1}\pi\int_{\mathbb{S}^{2d-1}}(\Phi(\vartheta))^{-2d/\beta} d\vartheta$. This immediately implies $(j/C)^{\frac{1}{2d}}\sim f^{-1}(\lambda_j)$, as $j\to\infty$. Note that \eqref{weyleq5} yields that $f$ is regularly varying of index $\beta$, i.e.,
$
f(\alpha\lambda)\sim \alpha^{\beta}f(\lambda),$ $\lambda\to\infty,$
uniformly for $\alpha>0$ on compacts of $(0,\infty)$. Using this,
$
\lambda_{j}=f((j/C)^{\frac{1}{2d}}(1+o(1)))\sim C^{-\frac{\beta}{2d}}f(j^{\frac{1}{2d}}),
$
which is (\ref{weyleq7.1}).
\end{proof}

\begin{proof}[Proof of Corollary \ref{Weylformulac1}] We only give the proof under the assumptions of Theorem \ref{Weylth1}, the proof of this corollary with the hypotheses from Theorem \ref{Weylth2} is similar and the details are therefore left to the reader. By Theorem \ref{Weylth1}, we only need to show that
$$
\int_{a(w)<\lambda}dw \sim C' \sigma(\lambda), \quad \lambda\to\infty,
$$
where $C'$ is given by (\ref{eqproofc1Weylextra}). We show that
$$
\limsup_{\lambda\to\infty}\frac{1}{\sigma(\lambda)}\int_{a(w)<\lambda}dw \leq C';
$$
one treats analogously the limit inferior to obtain the desired result and we thus omit the calculation. Fixing $\varepsilon>0$, using the lower bound from (\ref{weyleq3}) (choose $B_{\varepsilon}>Y$), polar coordinates, and (\ref{weyleq15}), we have
\begin{align*}
\limsup_{\lambda\to\infty}\frac{1}{\sigma(\lambda)}\int_{a(w)<\lambda}dw&\leq \limsup_{\lambda\to\infty} \frac{1}{\sigma(\lambda)} \int_{\mathbb{S}^{2d-1}}\int_{\{r|\: B_{\varepsilon}<r,\ a(r\vartheta)<\lambda\}} r^{2d-1}dr d\vartheta
\\
&
\leq \lim_{\lambda\to\infty} \frac{1}{\sigma(\lambda)} \int_{\mathbb{S}^{2d-1}}\int_{ B_{\varepsilon}}^{(1+\varepsilon)f^{-1}(\lambda/c_{\varepsilon})/\Phi(\vartheta)} r^{2d-1}dr d\vartheta
\\
&
= (1+\varepsilon)^{2d} C' \lim_{\lambda\to\infty}\frac{\sigma(\lambda/c_{\varepsilon})}{\sigma(\lambda)}=(1+\varepsilon)^{2d}C'.
\end{align*}
The result now follows by taking $\varepsilon\to0^{+}$.
\end{proof}

\begin{proof}[Proof of Theorem \ref{Weylth3}] The lower bound (\ref{weyleq9}) still applies to show (\ref{weyleq17}). Combining this with the asymptotic estimate (\ref{Weyleq16}), we obtain
\begin{equation}
\label{weyleqlast}
\int_{0}^{\infty}e^{-t\lambda }dN(\lambda)=\frac{1}{(2\pi)^{d}}\int_{\RR^{2d}}e^{-ta(w)}dw +o(\sigma(1/t)), \quad t\to0^{+}.
\end{equation}
When either (\ref{weyleq8}) or (\ref{weyleq811}) holds, fix $\nu>2d/\beta'$ and find $\lambda_0>0$ such that (\ref{weyleq15.2}) holds. For $0<t\leq 1/\lambda_{0}$, we deduce from (\ref{weyleq9}) that
\beq
\int_{\RR^{2d}}e^{-ta(w)}dw  &\leq& \int_{\mathbb{S}^{2d-1}} \int_{B\leq r} e^{-Ct f(r)}r^{2d-1}dr d\vartheta +O_{\nu}(1)\nonumber \\
&=&
\frac{\pi^{d}}{d!} \int_{0}^{\infty} e^{-\lambda} \sigma(\lambda/(Ct)) d\lambda +O_{\nu}(1).\label{kttrri119}
\eeq
If (\ref{weyleq8}) holds, then the monotonicity of $\sigma$ together with (\ref{weyleq15.2}) yields
\beqs
\int_{\RR^{2d}}e^{-ta(w)}dw&\leq& \frac{\pi^{d}}{d!}\int_{0}^{C} e^{-\lambda} \sigma(\lambda/(Ct)) d\lambda +\frac{\pi^{d}}{d!}\int_{C}^{\infty} e^{-\lambda} \sigma(\lambda/(Ct)) d\lambda +O_{\nu}(1)\\
&\leq& \frac{\pi^{d}}{d!}\sigma(1/t)+\frac{\pi^{d}}{d!}\sigma(1/t) \int_{0}^{\infty} e^{-\lambda} \left(\frac{\lambda}{C}\right)^{\nu} d\lambda +O_{\nu}(1)\\
&\leq& \frac{\pi^{d}}{d!}\sigma(1/t)\left(1+\frac{\Gamma(1+\nu)}{C^{\nu}}\right) +O_{\nu}(1).
\eeqs
Using (\ref{weyleqlast}) and keeping still $t\leq 1/\lambda_{0}$,
\begin{align*}
N(1/t)-N(0)&= \int_{0}^{1/t}dN(\lambda)\leq  e \int_{0}^{1/t} e^{-t\lambda }dN(\lambda)\leq e \int_{0}^{\infty} e^{-t\lambda }dN(\lambda)
\\
&
\leq\frac{e}{2^d d!}\left(1+\frac{\Gamma(1+\nu)}{C^{\nu}}\right)\sigma(1/t) (1+o_\nu(1)).
\end{align*}
Dividing through by $\sigma(1/t)$, taking the limit superior as $t\to0^{+}$, and letting then $\nu\to 2d/\beta'$, we obtain the estimate (\ref{weyleq10}). The lower bound \eqref{weyleq10.1} easily follows by inserting $\lambda=\lambda_{j}$ in (\ref{weyleq10}) and the fact $N(\lambda_j)\geq j$, $\forall j\in\NN$. If (\ref{weyleq811}) holds, we divide (\ref{kttrri119}) by $\sigma(1/t)$ and take the limit superior as $t\rightarrow 0^+$. Because of (\ref{weyleq15}) we have
\beqs
\limsup_{t\rightarrow 0^+}\frac{1}{\sigma(1/t)}\int_{\RR^{2d}}e^{-ta(w)}dw\leq \frac{\pi^d\Gamma(1+2d/\beta')}{d!C^{2d/\beta'}}.
\eeqs
Now, the same technique as before yields the rest of the assertions of the theorem.
\end{proof}

\section{Appendix}\label{app}

We collect here some important facts concerning symbolic calculus and the construction of parametrices for operators with symbols in $\Gamma_{A_p,\rho}^{*,\infty}(\RR^{2d})$. We start with the following continuity result.

\begin{proposition}\label{continuity}$($\cite[Proposition 3.1]{PP1}$)$
For each $\tau\in\RR$, the bilinear mapping $(a,\varphi)\mapsto
\Op_{\tau}(a)\varphi$,
$\Gamma_{A_p,\rho}^{*,\infty}(\RR^{2d})\times
\SSS^*(\RR^d)\rightarrow \SSS^*(\RR^d)$, is hypocontinuous and it
extends to the hypocontinuous bilinear mapping $(a,T)\mapsto
\Op_{\tau}(a)T$, $\Gamma_{A_p,\rho}^{*,\infty}(\RR^{2d})\times
\SSS'^*(\RR^d)\rightarrow \SSS'^*(\RR^d)$. The
mappings $a\mapsto \Op_{\tau}(a)$,
$\Gamma_{A_p,\rho}^{*,\infty}(\RR^{2d})\rightarrow
\mathcal{L}_b(\SSS^*(\RR^d), \SSS^*(\RR^d))$, $\Gamma_{A_p,\rho}^{*,\infty}(\RR^{2d})\rightarrow
\mathcal{L}_b(\SSS'^*(\RR^d), \SSS'^*(\RR^d))$ are continuous.
\end{proposition}

As we mentioned before, changing the quantisation always results in operators with symbols in $\Gamma_{A_p,\rho}^{*,\infty}(\RR^{2d})$ modulo $*$-regularising operators (see \cite{PP1,BojanP}).\\
\indent The composition of two Weyl quantisation is again a $\Psi$DO (modulo a $*$-regularising operator) with Weyl symbol ``given'' by their $\#$-product. More precisely

\begin{theorem}\label{weylq}$($\cite[Theorem 4.2]{PP1}$)$
Let $U_1,U_2\subseteq FS_{A_p,\rho}^{*,\infty}(\RR^{2d};B)$ be
such that $U_1\precsim f_1$ and $U_2\precsim f_2$ in
$FS_{A_p,\rho}^{*,\infty}(\RR^{2d};B)$ for some continuous
positive functions $f_1$ and $f_2$ with ultrapolynomial growth of
class $*$. Then:
\begin{itemize}
\item[$i)$] $U_1\#U_2\precsim f_1f_2$ in $FS_{A_p,\rho}^{*,\infty}(\RR^{2d};B)$.
\item[$ii)$] Let $V_k\precsim_{f_k} U_k$, with $\Sigma_k: U_k\rightarrow V_k$ the surjective mapping, $k=1,2$. There exists $R>0$, which can be chosen arbitrarily large, such that
    \beqs
    \left\{\Op_{1/2}\left(\Sigma_1(\ssum a_j)\right)\Op_{1/2}\left(\Sigma_2(\ssum b_j)\right)-\Op_{1/2}\left(R(\ssum a_j\# \ssum b_j)\right)\big|\right.
    \\
    \left.\ssum a_j\in U_1,\, \ssum b_j\in U_2\right\}
    \eeqs
    is an equicontinuous subset of $\mathcal{L}(\SSS'^*(\RR^d),\SSS^*(\RR^d))$ and
    \beq\label{krh1791}
    \left\{R(\ssum a_j\# \ssum b_j)\big|\, \ssum a_j\in U_1,\, \ssum b_j\in U_2\right\}\precsim_{f_1f_2}U_1\#U_2.
    \eeq
\end{itemize}
\end{theorem}

\begin{corollary}\label{corweylqu}$($\cite[Corollary 4.3]{PP1}$)$
Let $U_1,U_2\subseteq FS_{A_p,\rho}^{*,\infty}(\RR^{2d};B)$ with
$U_1\precsim f_1$ and $U_2\precsim f_2$ for some continuous
positive functions of ultrapolynomial growth of class $*$. For
$\sum_j a_j\in U_1$ and $\sum_j b_j\in U_2$ denote $\sum_j
c_{j,a,b}=\sum_j a_j\#\sum_j b_j\in U_1\# U_2$. Then, there exists
$R>0$, which can be chosen arbitrarily large, such that
\beqs
\left\{a^wb^w-c^w\big|\, a=R(\ssum a_j),\, b=R(\ssum b_j),\,
c=R(\ssum c_{j,a,b})\right\}
\eeqs
is an equicontinuous subset of $\mathcal{L}(\SSS'^*(\RR^d),\SSS^*(\RR^d))$ and \eqref{krh1791} holds.
\end{corollary}

\begin{remark}
Corollary \ref{corweylqu} is also applicable when $U_1$ and $U_2$ are bounded subsets of $\Gamma_{A_p,\rho}^{(M_p),\infty}(\RR^{2d};m)$ for some $m>0$ (resp. of $\Gamma_{A_p,\rho}^{\{M_p\},\infty}(\RR^{2d};h)$ for some $h>0$). In this case, the corollary reads: there exists $R>0$, which can be chosen arbitrary large, such that $\{a^wb^w-\Op_{1/2}(R(a\# b))|\, a\in U_1,\, b\in U_2\}$ is equicontinuous $*$-regularising set and $\{R(a\# b)|\, a\in U_1,\, b\in U_2\}$ is bounded in $\Gamma_{A_p,\rho}^{(M_p),\infty}(\RR^{2d};m)$ for some $m>0$ (resp. of $\Gamma_{A_p,\rho}^{\{M_p\},\infty}(\RR^{2d};h)$ for some $h>0$, cf. Lemma \ref{lemulgr117}).
\end{remark}

Hypoelliptic symbols have parametrices and hence they are globally regular; we can explicitly construct (the asymptotic expansions of) the parametrices.

\begin{proposition}\label{parametweyl}$($\cite[Proposition 5.2]{PP1}$)$
Let $a\in\Gamma^{*,\infty}_{A_p,\rho}(\RR^{2d})$ be hypoelliptic. Define $q_0(w)=a(w)^{-1}$ on $Q^c_B$ and inductively, for $j\in\ZZ_+$,
\beqs
q_j(x,\xi)=-q_0(x,\xi)\sum_{s=1}^j\sum_{|\alpha+\beta|=s}\frac{(-1)^{|\beta|}}{\alpha!\beta!2^s}\partial^{\alpha}_{\xi} D^{\beta}_x q_{j-s}(x,\xi) \partial^{\beta}_{\xi} D^{\alpha}_x a(x,\xi),\,\, (x,\xi)\in Q^c_B.
\eeqs
Then, for every $h>0$ there exists $C>0$ (resp. there exist $h,C>0$) such that
\beq\label{estofpara}
\left|D^{\alpha}_w q_j(w)\right|\leq C\frac{h^{|\alpha|+2j}A_{|\alpha|+2j}}{|a(w)|\langle w\rangle^{\rho(|\alpha|+2j)}},\,\, w\in Q^c_B,\,\alpha\in\NN^{2d},\, j\in\NN.
\eeq
If $B\leq 1$, then $(\sum_j q_j)\# a= \mathbf{1}$ in $FS_{A_p,\rho}^{*,\infty}(\RR^{2d};0)$. If $B>1$, one can extend $q_0$ to an element of $\Gamma^{*,\infty}_{A_p,\rho}(\RR^{2d})$ by modifying it on $Q_{B'}\backslash Q^c_B$, for $B'>B$. In this case $\sum_j q_j\in FS_{A_p,\rho}^{*,\infty}(\RR^{2d};B')$, $((\sum_j q_j)\# a)_k=0$ on $Q^c_{B'}$, $\forall k\in\ZZ_+$, and $((\sum_j q_j)\#a)_0-1=q_0a-1$ belongs to $\DD^{(A_p)}(\RR^{2d})$ (resp. $\DD^{\{A_p\}}(\RR^{2d})$).\\
\indent In particular, for $q\sim\sum_j q_j$ there exists $*$-regularising operator $T$ such that $q^wa^w=\mathrm{Id}+T$.
\end{proposition}

\begin{remark}\label{kts951307}
A similar construction yields $\tilde{q}\in\Gamma^{*,\infty}_{A_p,\rho}(\RR^{2d})$ such that $a^w\tilde{q}^w-\mathrm{Id}$ is $*$-regularising (see \cite[Subsection 6.2.1]{PP1} for more details). Knowing this, it is easy to prove that we can use the left parametrix $q^w$ as a right one as well, i.e. both $q^wa^w-\mathrm{Id}$ and $a^wq^w-\mathrm{Id}$ are $*$-regularising.
\end{remark}

\begin{remark}\label{ktv957939}
For hypoelliptic $a\in\Gamma^{*,\infty}_{A_p,\rho}(\RR^{2d})$, we can construct a parametrix $q$ out of $\sum_j q_j\in FS_{A_p,\rho}^{*,\infty}(\RR^{2d};B')$ in a specific way. Namely, applying Corollary \ref{corweylqu} to $(\sum_j q_j)\# a$ together with (\ref{estofpara}) and Proposition \ref{subordinate}, we conclude the existence of $R>0$ and a $*$-regularising operator $T$ such that $q^wa^w=\mathrm{Id}+T$, where $q=R(\sum_j q_j)\in \Gamma^{*,\infty}_{A_p,\rho}(\RR^{2d})$ satisfies the following conditions: there exist $B''\geq B'$ and $c'',C''>0$ such that
\beq\label{ktr991509}
c''/|a(w)|\leq |q(w)|\leq C''/|a(w)|,\,\, \forall w\in Q^c_{B''},
\eeq
and for every $h>0$ there exists $C>0$ (resp. there exist $h,C>0$) such that
\beq\label{ktl997133}
\left|D^{\alpha}_w q(w)\right|\leq Ch^{|\alpha|}A_{\alpha}|a(w)|^{-1}\langle w\rangle^{-\rho|\alpha|},\,\, w\in Q^c_{B''},\,\alpha\in\NN^{2d},\, j\in\NN.
\eeq
In particular, $q$ is hypoelliptic. This estimate leads to the following simple observation. Assume that $a$ is hypoelliptic and $|a(w)|\rightarrow \infty$ as $|w|\rightarrow\infty$ and let $q$ be the parametrix for $a$ constructed above. Take $\psi\in \DD^{(A_p)}(\RR^{2d})$ (resp. $\psi\in\DD^{\{A_p\}}(\RR^{2d})$) such that $0\leq \psi\leq 1$, $\psi=1$ on a compact neighbourhood of $Q_{B''}$ and $\psi=0$ on the complement of a slightly larger neighbourhood. Then, for each $n\in\ZZ_+$, the function $b_n(w)=q(w)\psi(w/n)$ is in $\DD^{(A_p)}(\RR^{2d})$ (resp. in $\DD^{\{A_p\}}(\RR^{2d})$) and hence $b_n^w$ is $*$-regularising for each $n\in\ZZ_+$. Employing the fact $|a(w)|\rightarrow \infty$ as $|w|\rightarrow\infty$ together with (\ref{ktl997133}), one easily verifies that $b_n\rightarrow q$ in $\Gamma^0_{\rho}(\RR^{2d})$ and hence $b_n^w\rightarrow q^w$ in $\mathcal{L}_b(L^2(\RR^d),L^2(\RR^d))$ (see \cite[Theorem 1.7.14, p. 58]{NR}). As $b_n^w$, $n\in\ZZ_+$, are compact operators on $L^2(\RR^d)$, so is $q^w$.
\end{remark}

\end{document}